%% file: Intersections.tex
\documentclass[11pt]{aart}
\usepackage[margin=1in]{geometry}                
\usepackage{multind}
\usepackage{multicol}

\usepackage
[pdfauthor={Aaron Michael Silberstein},
 pdftitle={Anabelian Intersection Theory I: The Conjecture of Bogomolov-Pop and Applications},
 bookmarks=false,
 hyperindex]
{hyperref}

\input{headerbeef}

\renewcommand{\Im}{\operatorname{Im}}
\usepackage{comment}
\excludecomment{comment}
\excludecomment{thesisonly}
\title{Anabelian Intersection Theory I: The Conjecture of Bogomolov-Pop and Applications}
\author{Aaron Michael Silberstein\footnote{University of Pennsylvania, Philadelphia, PA 19104.  Email: \texttt{aaronsil@math.upenn.edu}}}

\begin{document}
\maketitle
\section{Statement of Results}
A.~Grothendieck first coined the term ``anabelian geometry'' in a letter to G.~Faltings \cite{GrothLTF} as a response to Faltings' proof of the Mordell conjecture and in his celebrated \textit{Esquisse d'un Programme} \cite{GrothEsq}.  The ``yoga'' of Grothendieck's anabelian geometry is that if the \'etale fundamental group $\poe(X, \overline{x})$ of a variety $X$ at a geometric point $\overline{x}$ is rich enough, then it should encode much of the information about $X$ as a variety; such varieties $X$ are called \textbf{anabelian in the sense of Grothendieck}, and have the property that two anabelian varieties have isomorphic \'etale fundamental groups if and only if they are isomorphic; and that the isomorphisms between their \'etale fundamental groups are precisely the isomorphisms between the varieties.  Grothendieck did not specify how much extra information should be encoded, and there is currently not a consensus on the answer.  An \textbf{anabelian theorem (or conjecture)} is a theorem (or conjecture) which asserts that a class of varieties are anabelian.

Grothendieck wrote in \cite{GrothLTF} about a number of anabelian conjectures, one regarding the moduli of curves, defined over global fields (which is still open); one regarding hyperbolic curves, defined over global fields; and a birational anabelian conjecture, which asserts that $\Spec$ of finitely-generated, infinite fields are anabelian (in this case, we say the fields themselves are anabelian).  The anabelian conjecture for hyperbolic curves was proved in the 1990's by A.~Tamagawa and S.~Mochizuki (\cite{TamagawaGrothAffCurves}, \cite{Moc}).  The birational anabelian conjecture for finitely-generated, infinite fields is a vast generalization of the pioneering Neukirch-Ikeda-Uchida theorem for global fields (\cite{NeukVals}, \cite{Uch}, \cite{Ikeda}, \cite{NeukAnab}), and is now a theorem due to F.~Pop \cite{PopAnab}.

Grothendieck remarked that ``the reason for [anabelian phenomena] seems\dots to lie in the extraordinary \textit{rigidity} of the full fundamental group, which in turn springs from the fact that the (outer) action of the `arithmetic' part of this group\dots is extraordinarily strong'' \cite{GrothLTF}.

F.~Bogomolov had the surprising insight \cite{BogTwoConj} that as long as the dimension of a variety is $\geq 2$, anabelian phenomena can be exhibited --- at least birationally --- even over an algebraically closed field, \textit{even in the complete absence of the ``arithmetic'' part of the group Grothendieck referenced.}

Given a field $K,$ we let $G_K$ denote the absolute Galois group of $K$, the profinite group of field automorphisms of its algebraic closure $\overline{K}$ (see \cite{CohNumFields} for more details).  Given two fields $F_1$ and $F_2$, we let $\Isom^{i}(F_1, F_2)$ denote the set of isomorphisms between the pure inseparable closures of $F_1$ and $F_2$, up to Frobenius twists.  Given two profinite groups $\Gamma_1$ and $\Gamma_2$, we let $\Isomext(\Gamma_1, \Gamma_2)$ denote the set of equivalence classes of continuous isomorphisms from $\Gamma_1$ to $\Gamma_2$, modulo conjugation by elements of $\Gamma_{2}$.  There is a canonical map
\begin{formula}
\varphi_{F_{1}, F_{2}}: \Isom^{i}(F_{1}, F_{2})\rightarrow \Isomext(G_{F_{2}}, G_{F_{1}})
\end{formula}
which, in general, is neither injective nor surjective.

The birational theory of a variety of dimension $n$ over $K$ is encoded in its field of rational functions, and every field finitely-generated over $K$ and of transcendence degree $n$ arises as the field of rational functions of a $K$-variety of dimension $n$. F.~Pop, developing Bogomolov's insight, conjectured an anabelian theorem for fields, finitely-generated and of transcendence degree $n\geq 2$ over an algebraically closed field $k$.  We complete the proof of:
\begin{theorem}[The Conjecture of Bogomolov-Pop for $k = \overline{\mathbb{Q}}, \overline{\mathbb{F}}_{p}$]\label{BogPopConj}
Let $F_{1}$ and $F_{2}$ be fields finitely-generated and of transcendence degree $\geq 2$ over $k_{1}$ and $k_{2}$, respectively, where $k_{1}$ is either $\overline{\mathbb{Q}}$ or $\overline{\mathbb{F}}_{p}$, and $k_{2}$ is algebraically closed.  Then $\varphi_{F_{1}, F_{2}}$ is a bijection.  Thus, function fields of varieties of dimension $\geq 2$ over algebraic closures of prime fields are anabelian.
\end{theorem}
In \cite{PopIOM}, Pop proved that if $G_{F_{1}} \simeq G_{F_{2}}$ then $F_{1}$ and $F_{2}$ have the same characteristic and transcendence degree.  Thus, the conjecture reduces to the case when $F_{1}$ and $F_{2}$ are of the same characteristic and transcendence degree.  Bogomolov and Tschinkel \cite{BogTschRec}  provide a proof in the case of transcendence degree $= 2$ when $k = \overline{\mathbb{F}}_{p}$. Pop proved that $\varphi$ is a bijection when $F_{1}$ has transcendence degree $\geq 2$ and $k = \overline{\mathbb{F}}_{p}$ \cite{PopBog1}; and when $F_{1}$ has transcendence degree $\geq 3$ and $k = \overline{\mathbb{Q}}$ \cite{PopBog2}.    We prove the missing case:

\begin{theorem}[The Birational Anabelian Theorem for Surfaces over $\overline{\mathbb{Q}}$]\label{BirationalAnabelianTheoremForSurfaces} Let $F_{1}$ and $F_{2}$ be fields finitely-generated and of transcendence degree $2$ over $\overline{\mathbb{Q}}$.  Then $\varphi_{F_{1}, F_{2}}$ is a bijection.
\end{theorem}

The proof of \autoref{BirationalAnabelianTheoremForSurfaces} is substantially different in structure from the other cases of \autoref{BogPopConj}.  They both have the same starting point, two theorems due to Pop from \cite{PopBog2} and described in \autoref{models}:
\begin{enumerate}
\item Given a subgroup $\Gamma \subseteq G_{F}$ there is a profinite group-theoretic recipe (\autoref{ValuationsRecipe}) which determines whether or not $\Gamma$ is a decomposition or inertia group of a Parshin chain (\autoref{def:ParshinChain}).
\item Given a collection $\caS = \{T_{v}\}$ of inertia groups of rank-$1$ Parshin chains (which are group-theoretically definable by \autoref{ValuationsRecipe}), there is a recipe to determine whether there is a model $X$ of $F$ --- that is, a smooth variety with function field $F$ --- for which $\caS$ is the set of inertia groups of Weil prime divisors centered on $X$.  In this case, $\caS$ is called a \textbf{geometric set of prime divisors} (\autoref{def:geometricset}).
\end{enumerate}
Previous results took data such as these and reconstructed $F$ directly, in a process which we now term \textbf{birational reconstruction}.  However, in our approach, we instead take the pair $(G_{F}, \caS)$ and reconstruct a model $\caM(\caS)$ of $F$ for which $\caS$ is the collection of inertia subgroups of all prime divisors on $\caM(\caS)$.  We obtain a description of the geometry of $\caM(\caS)$ without first reconstructing $F$, and we call this approach \textbf{geometric reconstruction}.  The main tool is the ability to interpret intersection theory on $\caM(\caS)$ using only group theoretic recipes applied to $\caS$ and $G_{F},$ without any knowledge of $\caM(\caS)$ other than its existence; this technique is the \textbf{anabelian intersection theory} of the title.  In \autoref{pf:BiratAnabThm} we show how \autoref{BirationalAnabelianTheoremForSurfaces} follows from the geometric reconstruction results which take up the bulk of the paper.  A generalization of the geometric reconstruction technique in transcendence degree $\geq 2$ for arbitrary characteristic will come in a sequel to this paper.

Grothendieck also asked for a ``purely geometric'' description of the group $G_{\mathbb{Q}}.$ Anabelian geometry over $\overline{\mathbb{Q}}$ gives such a description.\footnote{In fact, Y.~Ihara asked a more refined question, which T.~Oda and M.~Matsumoto raised to a conjecture: is $G_{\mathbb{Q}}$ exactly the outer automorphisms of the \'etale fundamental group functor from the category of $\mathbb{Q}$-varieties to profinite groups?  Pop showed \cite{PopIOM} how birational anabelian theorems over $\overline{\mathbb{Q}}$ can be used to provide proofs of this conjecture.  Thus, \autoref{BirationalAnabelianTheoremForSurfaces} gives us a new proof of the Question of Ihara/Conjecture of Oda-Matsumoto.  We will elaborate on applications of geometric reconstruction to refinements of the Question of Ihara/Conjecture of Oda-Matsumoto in a later paper.} 

Let $X$ be an irreducible, geometrically integral, algebraic variety of dimension $\geq 2$ defined over $\overline{\mathbb{Q}}$ so that $X$ has no birational automorphisms defined over $\overline{\mathbb{Q}}$. We call such an $X$ \textbf{birationally, geometrically rigid}.  Let $k$ be the intersection of all subfields of $\overline{\mathbb{Q}}$ over which a variety birationally equivalent to $X$ is defined.  Let $\overline{\mathbb{Q}}(X)$ be the rational function field of $X$.  By our assumptions, $\overline{\mathbb{Q}}(X)$ is a field, finitely generated over $\overline{\mathbb{Q}}$.  \autoref{GGInvLimofPi1} shows that $G_{\overline{\mathbb{Q}}(X)}$ is a geometric object: it is the inverse limit of the profinitely-completed fundamental groups of all complements of $\overline{\mathbb{Q}}$-divisors in $X$.

\begin{theorem}[The Geometric Description of Absolute Galois Groups of Number Fields.]\label{outerauts}
The natural map
\begin{equation}
G_{k}\rightarrow \Out_{\cont}(G_{\overline{\mathbb{Q}}(X)})
\end{equation}
is an isomorphism, and the compact-open topology induces the topology on $G_{k}$ induced by its structure as a Galois group.
\end{theorem}
\autoref{outerauts} answers Grothendieck's question; this is analogous to a theorem of I.~Bumagin and D.~Wise, which realizes any countable group as the outer automorphism group of a finitely generated group \cite{WiseBumOutGroup}. In \autoref{sec:AutomorphismGroups}, we write down, for any number field $k$, an infinite family of explicit varieties $X$ which satisfy the hypotheses of \autoref{outerauts}.

Any group $G_{\overline{\mathbb{Q}}(X)}$ which satisfies the hypotheses of \autoref{outerauts} is infinitely-generated and infinitely-presented.  It is then natural to ask whether absolute Galois groups of number fields are outer automorphism groups of finitely-generated, finitely-presented, profinite groups of geometric provenance.

The ``smallest group'' currently considered to be a candidate to give such a geometric representation is the group $\widehat{GT}$, whose study was initiated by Drinfel'd \cite{Drinfeld} and Ihara \cite{IharaBraids}.  $\widehat{GT}$ admits an injective group homomorphism 
\begin{equation}
\rho: G_{\mathbb{Q}}\rightarrow \widehat{GT}.
\end{equation}
It is not known whether $\rho$ is surjective.  $\widehat{GT}$ is a much-studied yet poorly-understood object; we review in \autoref{sec:GT} the theory we will need, and refer to \cite{LochakSchneps} for a more exhaustive survey.

We denote by $\caM_{0,5}$ the moduli space of genus $0$ curves with $5$ distinct, marked, ordered points.  As part of the construction of $\widehat{GT}$, we have an injection
\begin{formula}\label{etadefinition}\eta: \widehat{GT}\rightarrow \Out(\poe(\caM_{0,5})),\end{formula}
such that the composition
\begin{equation}\eta\circ \rho: G_{\mathbb{Q}}\rightarrow \Out(\poe(\caM_{0,5}))\end{equation}
is the injection induced by the theory of the fundamental group.

We prove a necessary and sufficient criterion for an element of $\widehat{GT},$ as determined by its image under $\eta$, to be in the image of $\rho$.  $\caM_{0,5}$ is birationally equivalent to $\mathbb{P}^{2}$, which gives the scheme-theoretic inclusion
\begin{equation}
\gamma: \Spec(\overline{\mathbb{Q}}(x,y))\rightarrow \caM_{0,5}
\end{equation}
of the generic point of $\caM_{0,5}$.  The \'etale fundamental group functor then gives a continuous surjection of profinite groups
\[
\gamma_{*}: G_{\overline{\mathbb{Q}}(x,y)}\rightarrow \poe(\caM_{0,5}),
\]
well-defined up to conjugation by an element of $\poe(\caM_{0,5})$.
\begin{theorem}\label{LiftingCondition}
Let $\alpha \in \im \eta.$  Then $\alpha \in \im \eta\circ \rho$ if and only if $\alpha$ satisfies the following lifting condition: there exists an automorphism
\begin{equation}
\tilde{\alpha}: G_{\overline{\mathbb{Q}}(x,y)}\rightarrow G_{\overline{\mathbb{Q}}(x,y)}
\end{equation}
so that the diagram
\begin{center}
\mbox{\xymatrix{
G_{\overline{\mathbb{Q}}(x,y)} \ar[r]^{\tilde{\alpha}} \ar[d]^{\gamma_{*}}& G_{\overline{\mathbb{Q}}(x,y)} \ar[d]^{\gamma_{*}} \\
\poe(\caM_{0,5})\ar[r]^{\alpha} & \poe(\caM_{0,5})
}}
\end{center}
commutes up to inner automorphisms.
\end{theorem}
This is the first necessary and sufficient \textit{geometric} condition for an element of the Grothendieck-Teichm\"uller group to lie in the image of $G_{\mathbb{Q}}$.

In \autoref{sec:AutomorphismGroups} we provide explicit examples of surfaces which satisfy the hypotheses of \autoref{outerauts}, and we conclude with the proof of \autoref{LiftingCondition}.

\section{The Geometric Interpretation of Inertia and Decomposition Groups of Parshin Chains}\label{sec:GeomInterpretation}
\begin{definition} \label{def:functionfield}Let $F$ be a field finitely generated over some algebraically closed field $K$ of characteristic zero; such a field will be called a \textbf{function field}\index{definitions}{Function field}.  Note that the field $K$ is determined by $F$ (for instance, its multiplicative group is the set of all divisible elements in the multiplicative group of $F$); it will be denoted by $K(F)$, and be called the \textbf{field of constants} of $F$\index{definitions}{Field of constants}.  The transcendence degree of $F$ over $K(F)$ will be called the \textbf{dimension}\index{definitions}{Dimension} of $F$.  We will denote by $G_F$ the absolute Galois group of $F$\index{notation}{GF@$G_{F}$} and $\overline{F}$\index{notation}{Fbar@$\overline{F}$} the algebraic closure of $F$.  
\end{definition}
For general theory of valuations, including proofs of the algebraic theorems cited without proofs, see \cite{JardenFried}. We will also use results from \cite{GAGA} and \cite{SGA1} with impunity.  

\begin{definition}\label{def:valuation}A \textbf{valuation} $v$ on $F$ is an ordered group $(vF, \leq),$\index{notation}{vF@$vF$} called the \textbf{value group} of $v$\index{definitions}{Value group}, along with a surjective map\index{definitions}{Valuation}
\begin{equation}\label{not:valuegroup}
v: F\rightarrow vF \cup \{\infty\}
\end{equation}\index{notation}{v@$v$}
which satisfies
\begin{enumerate}
\item $v(x) = \infty$ if and only if $x = 0$.
\item $v(xy) = v(x) + v(y)$ (here, we define $\infty + g = \infty$ for all $g \in vF \cup \{\infty\}$.)
\item $v(x+y) \geq \min \{v(x), v(y)\},$ where we extend the ordering $\leq$ to $vF \cup \{\infty\}$ by $g \leq \infty$ for all $g \in vF$.
\end{enumerate}
A valuation $v$ gives rise to a \textbf{valuation ring}\index{definitions}{Valuation ring} $\caO_{v}$\index{notation}{Ov@$\caO_{v}$}, which is the set of all $x \in F$ such that $v(x) \geq 0$.  $\caO_{v}$ is integrally closed in $F$ and local, and we call its maximal ideal $\mathfrak{m}_{v}$\index{notation}{mv@$\mathfrak{m}_{v}$}.  We then define the \textbf{residue field}\index{definitions}{Residue field} to be
\begin{equation}\label{def:residuefield}
Fv = \caO_{v}/\mathfrak{m}_{v},
\end{equation}\index{notation}{Fv@$F_{v}$}
A subring $R\subseteq F$ is a valuation ring $\caO_{w}$ for some valuation $w$ if and only if for every $x \in F^{\times}$ either $x \in R$ or $x^{-1}\in R$.  Therefore, equivalently, we may define a valuation $v$ by its \textbf{place}\index{definitions}{Place}
\begin{equation}\label{def:valuationplace}
p_{v}: F\rightarrow Fv \cup \{\infty\}
\end{equation}\index{notation}{pv@$p_{v}$}
where $\caO_{v}$ is mapped to its reduction mod $\mathfrak{m}_{v}$ and $F\setminus \caO_{v}$ is mapped to $\infty$.  Any map from a field $F$ to another field $L$ which is a ring homomorphism ``with $\infty$'' thus gives rise to a valuation on $F$.
\end{definition}
\begin{definition}\label{def:equivalenceofvaluations}\index{definitions}{Valuation!Equivalence of}
  Two valuations will be called \textbf{equivalent} if and only if they have the same valuation ring.
  \end{definition}
\begin{definition} If $v$ is a valuation on $F$ and $w$ is a valuation on $Fv$, then we may define a valuation $w\circ v$ on $F$ by considering the composition
\begin{equation}
p_{w}\circ p_{v}: F\rightarrow (Fv)w \cup \{\infty\}
\end{equation}
as a place map on $F$.  This valuation is called the \label{def:valuationcomposition}\index{definitions}{Valuation!Composition}\textbf{composition} of $w$ with $v$.
\end{definition}
\begin{definition}\label{def:model}\index{definitions}{Model}\index{definitions}{Model!Normal}\index{definitions}{Structure map}
A \textbf{model} $X$ of a function field $F$ is a smooth, connected $K(F)$-scheme of finite-type with a  map
\begin{equation}\label{not:structuremap}\index{notation}{sX@$s_{X}$} s_{X}: \Spec F\rightarrow X,
\end{equation}
the \textbf{structure map} of the model, which identifies $F$ with the field of rational $K(F)$-functions on $X$.  A \textbf{normal model} is a model with the requirement of smoothness replaced by normality.
\end{definition}
\begin{definition}\label{not:BirF} By virtue of the structure map, the models form a full subcategory of the category of schemes under $\Spec F$.  We define an \textbf{$F$-morphism} \index{definitions}{F morphism@$F$-morphism} of models to be a morphism of varieties under $\Spec F$, and $\Bir(F)$ \index{notation}{BirF@$\Bir(F)$} the full subcategory of varieties under $\Spec F$ whose objects are precisely the models of $F$.
\end{definition}
\begin{definition}\label{def:center}\index{definitions}{Center} \index{notation}{vv@$\lvert v\rvert$} We say that a valuation $v$ on $F$ \textbf{has a center on} or \textbf{is centered on} $X$ if $X$ admits an affine open subset $\Spec A$ such that $A\subset \caO_{v}$, the valuation ring of $v$.  Let $R\subset F$ be a subring giving an affine open $\Spec R\subseteq X$.  Then the \textbf{center} of $v$ on $\Spec R$ is the Zariski closed subset of $X$ defined as $Z(\mathfrak{m}_{v}\cap R);$ the center of $v$ on $X$ is the union of the centers of $v$ on $\Spec R$ as $\Spec R$ ranges over all affine opens of $X$.  We denote by $|v|$ the center of $v$.
\end{definition}
\begin{definition}
A \textbf{prime divisor}\index{definitions}{Divisor, prime} $v$ on $F$ is a discrete valuation trivial on $K(F)$ such that \begin{equation}\label{def:notranscendencedefect}\trdeg_{K(F)} Fv = \trdeg_{K(F)} F - 1.\end{equation}\index{definitions}{Transcendence defect}
(This condition is very important in birational anabelian geometry in general, and we say that $v$ has \textbf{no transcendence defect}; see \cite{PopAnab} for more details.)
\end{definition}
\begin{definition}A \textbf{rank-$1$ Parshin chain} \index{definitions}{Parshin chain} for $F$ is a prime divisor.  A \textbf{rank-$i$ Parshin chain} is a composite $w\circ v$, where $v$ is a rank-$(i-1)$ Parshin chain, and $w$ is a prime divisor on $Fv$.  \end{definition}
\begin{definition}\label{def:ParshinChain}
We denote by $\Par_{i}(F)$\index{notation}{Pari@$\Par, \Par_{i}$|(} the collection of $i$-Parshin chains for $F,$ where $i \leq \trdeg_{K(F)} F$.  Given a rank-$k$ Parshin chain $v$ we denote by $\Par_{i}(v)$ the collection of rank-$i$ Parshin chains of the form $w\circ v$.  $\Par_{i}(v)$ is empty if $i < k$; is $\{v\}$ if $i = k$; and is infinite if $i > k$.  If $S\subseteq \Par_{k}(F)$ we define \begin{equation}
\Par_{i}(S) = \bigcup_{v \in S} \Par_{i}(v).
\end{equation}
We also let
\begin{equation}
\Par(S) = \bigcup_{i} \Par_{i}(S)\text{ and }\Par(F) = \bigcup_{i} \Par_{i}(F).
\end{equation}
\index{notation}{Pari@$\Par, \Par_{i}$|)}
\end{definition}
\begin{example}To describe the rank-$1$ Parshin chains on $F$, where $F$ is the function field of a \textbf{surface} over $\overline{\mathbb{Q}}$, we consider the set of all pairs $(X, D)$ where $X$ is a \textbf{proper} model of $F$ and $D$ is a prime divisor on $X$.  On this collection, we have an equivalence relation generated by the relation $\sim$, where we say \[(X, D)\sim (X', D')\] if there exists a rational map
\[
\varphi: X\rightarrow X'
\]
respecting the structure maps of the models $X$ and $X'$ such that $D$ is mapped birationally to $D'$ by $\varphi$.

The rank-$2$ Parshin chains are then equivalence classes $(X, D, p)$ where $D$ is a prime divisor on $D$ and $p$ is a smooth point on $D$, and the equivalence relation $\sim$ is now generated by saying \[(X, D, p) \sim (X', D', p')\] as long as there is a rational map 
\[
\varphi: X\rightarrow X'
\]
respecting the structure maps of the models $X$ and $X'$ so that $D$ is mapped birationally to $D'$ by $\varphi$ and $p$ is mapped to $p'$ by $\varphi$.
\end{example}
Given an algebraic extension $L|F$ every valuation extends to $L$, though not necessarily uniquely.
\begin{definition}We define \index{notation}{Xv@$\caX_{v}$} $\caX_v(L|F)$ to be the set of valuations on $L$ which restrict to $v$ on $F$.  If $L|F$ is Galois, then $\Gal(L|F)$ acts transitively on $\caX_v(L|F)$.  For any Galois extension $L|F$ and $\tilde{v} \in \caX_{v}(L|F)$ we define the \textbf{decomposition group}\index{definitions}{Decomposition group} \begin{equation}\index{notation}{Dvt@$D_{\tilde{v}}$}
D_{\tilde{v}}(L|F)\eqdef \left\{\sigma \in \Gal(L|F) \mid \sigma(\caO_{\tilde{v}}) = \caO_{\tilde{v}}\right\}.\end{equation} Each $D_{\tilde{v}}(L|F)$ has a normal subgroup, the \textbf{inertia group}\index{definitions}{Inertia group}
$T_{\tilde{v}}(L|F)$\index{notation}{Tvt@$T_{\tilde{v}}$}, defined as the set of elements which act as the identity on $L\tilde{v}$.
\end{definition}
We have a short-exact sequence, the \textbf{decomposition-inertia exact sequence}
\begin{equation}\label{def:decompositioninertiaexactsequence}
1 \rightarrow T_{\tilde{v}}(L|F) \rightarrow D_{\tilde{v}}(L|F) \rightarrow \Gal(L\tilde{v}|Fv)\rightarrow 1.
\end{equation}
If $\tilde{v}_{1}, \tilde{v}_{2} \in \caX_{v}(L|F)$ and $\tilde{v}_{1} = \sigma \tilde{v}_{2}$ for some $\sigma \in \Gal(L|F)$, then
\[
D_{\tilde{v}_{1}}(L|F) = \sigma^{-1}D_{\tilde{v}_{2}}(L|F)\sigma\text{ and }T_{\tilde{v}_{1}}(L|F) = \sigma^{-1} T_{\tilde{v}_{2}}(L|F)\sigma.
\]
Thus, all $D_{\tilde{v}}(L|F)$ and $T_{\tilde{v}}(L|F)$, respectively, are conjugate for a given $v$, and when the lift is not important, we denote some element of the conjugacy class of subgroups by $D_v(L|F)$ \index{notation}{Dv@$D_{v}$} and $T_{v}(L|F)$,\index{notation}{Tv@$T_{v}$} respectively.  We define $D_v$ and $T_v$, respectively, to be $D_v(\overline{F}|F)$ and $T_v(\overline{F}|F)$.

For any Galois extension $L|F$, a valuation $v$ on $F$ and a valuation $w$ on $Fv$, we may choose $\tilde{v}\in \caX_{v}(L|F)$ and $\tilde{w}\in \caX_{w}(L\tilde{v}|Fv)$.  There is then a natural short exact sequence, the \textbf{composite inertia sequence}:
\begin{equation}\label{def:compositeinertiasequence}
1\rightarrow T_{\tilde{v}}(L|F)\rightarrow T_{\tilde{w}\circ \tilde{v}}(L|F)\rightarrow T_{\tilde{w}}(L\tilde{v}|Fv) \rightarrow 1,
\end{equation}
where $T_{w}\subseteq G_{Fv}$, for any composite of valuations.  Thus, if $T_{\tilde{v}}(L|F)$ is trivial, \[T_{\tilde{w}} \simeq T_{\tilde{w}\circ \tilde{v}}.\]

We will use three different types of fundamental groups, with the following notation:\index{notation}{pione@$\pi_{1}$}
\begin{enumerate}
\item $\pi_{1}^{\top}$ will denote the topological fundamental group, the fundamental group of a fiber functor on the category of topological covers of a topological space; covering space theory \cite{hatcher} shows that $\pi_{1}^{\top}$ can be computed using based homotopy classes of maps into $S^{1}$; this is the original fundamental group considered by Poincar\'e \cite{Poincare}.
\item $\hat{\pi}_{1}$ will denote the profinite completion of $\pi_{1}^{\top}$, the fundamental group of a fiber functor on the category of \textit{finite} topological covers of a topological space.
\item $\poe$ will denote the \'etale fundamental group.
\end{enumerate}
For a normal variety $X$ over an algebraically closed subfield of $\mathbb{C}$, one has an equivalence between the category of finite, \'etale covers of $X$ and the finite, unramified covers of $X^{\anal}$, its corresponding analytic space over $\mathbb{C}$, by \cite{GrauertRemmert}.  This leads immediately to the
\begin{theorem}[Comparison Theorem]\label{thm:fundamentalgroupscomparison}
Let $X$ be a normal variety over $\mathbb{C}$ and $x \in X(\mathbb{C}).$ There is a canonical isomorphism
\begin{equation}\label{def:comparisontheoremcanonicalisomorphism}
\hat{\pi}_1(X^{\anal}, x) \simeq \poe(X, x).
\end{equation}
\end{theorem}

Every known computation of nonabelian fundamental groups of varieties factors through this comparison theorem; in characteristic $p$, for instance, this is combined with Grothendieck's specialization theorem \cite[X.2.4]{SGA1} to obtain information about fundamental groups.

Let now $K = \mathbb{C}$ and $F$ be a function field over $\mathbb{C}$.
Then we have the following interpretation of $D_v$ when $v$ is a prime divisor.  First, $v$ is the valuation associated to a Weil prime divisor on some normal model $X$ of $F$ --- that is, a normal variety with function field $F$, considered as a $\mathbb{C}$-scheme.  Given $X$, there is a corresponding normal analytic space $X^{\anal}$\index{notation}{Xan@$X^{\anal}$}.  Let $X$ be a model of $F$ on which $|v|$ is a prime divisor.
\begin{example}
The exceptional divisor $E$ on $\Bl_p(X)$, the blowup at some closed point $p$ of $X$, gives a prime divisor on $F$ but its center on $X$ is not codimension $1$ and so is not centered as a prime divisor on $X$.
\end{example}
Let $D'\subseteq X$ be the nonsingular locus of $|v|;$ notice that the underlying topological space of $D'$ is connected, as $v$ is a prime divisor.  \begin{definition}Let $\caN$ then be a normal disc bundle for (equivalently, a tubular neighborhood of) $D'$ and \[
\caT = \caN \setminus D'
\] the complement of $D'$ in its normal disc bundle $\caN$, which admits the \textbf{normal bundle fiber sequence}
\begin{equation}\label{not:normalbundlefibresequence}\mbox{\xymatrix{
1\ar[r] & \caF \ar[r]^-\iota & \caT \ar[r]^-\pi & D' \ar[r] &  1,
}}
\end{equation}
where $\caF$ is one of the fibers.  
\end{definition}
Let $p$ be a point on $\caF$.  Note that $\caF$, like all fibers of $\pi$, is a once-punctured disk and thus is homotopy equivalent to a circle.  There is a surjection
\begin{equation}
\rho: G_F \rightarrow \poe(X, p) \simeq \hat{\pi}_1(X^{\anal}, p)
\end{equation}
(proof: each normal \'etale cover of $X$ gives a normal extension of its field of functions) whose kernel we will define to have fixed field $F_X$, and the following commutative diagram:\\
\begin{equation}\label{ses:decompositioninertia}
\mbox{\xymatrix{
1 \ar[r] & \hat{\pi}_1(\caF, p) \ar[r]^-\iota \ar[dr] & \hat{\pi}_1(\caT, p) \ar[r]^-\pi\ar[d] & \hat{\pi}_1(D', \pi(p))\ar[r] & 1 \\
& & \poe(X, p). & & 
}}
\end{equation}
Then we have
\begin{proposition}[The Geometric Theory of Decomposition and Inertia Groups] \label{thm:geometrictheoryofdecompositionandinertia}
In the short exact sequence \ref{ses:decompositioninertia}:
\begin{enumerate}
\item The top row is a central extension of groups, as the normal bundle is complex-oriented.
\item The image of $\hat{\pi}_1(\caF, p)$ in $\poe(X, p)$ is a $T_v(F_X|F)$.
\item The image of $\hat{\pi}_1(\caT, p)$ in $\poe(X, p)$ is a $D_v(F_X|F)$.
\item $\hat{\pi}_1(D', \pi(p))$ is a quotient of $G_{Fv}$, corresponding to covers of $D'$ pulled back from covers of $X$.
\end{enumerate}
\end{proposition}
We will denote by $t_{v}$\index{notation}{tv@$t_{v}$} a generator of $\pi_{1}^{\top}(\caF, p) \subseteq \hat{\pi}_{1}(\caF, p),$ as well as any of its images in $\poe(X, p)$ or $T_{v}(F_{X}|F)$.
\begin{definition}\label{def:meridian}\index{definitions}{Meridian}
We refer to such a $t_{v}$ as a \textbf{meridian} of $v$.
\end{definition}
Each meridian is almost unique --- its inverse also gives a meridian of $v$, albeit ``in the opposite direction''.  This should be viewed as a ``loop normal to or around $|v|$''.  Its image generates $\hat{\pi}_{1}(\caF, p)$.  In general, if we are working in a situation in which we do not specify a basepoint, the meridian becomes defined only up to conjugacy.
\begin{definition}\label{not:abelianization} Let $\Gamma$ be a subgroup of a group $\Pi$.  Then the abelianization functor gives a map \begin{equation} \ab: \Gamma^{\ab}\rightarrow \pi_{1}^{\ab}.\end{equation}
We denote by $\Gamma^{a}$ the image of $\ab$\index{notation}{a@$^{a}$}.  In particular, given $v$ a valuation, and $\Pi$ a quotient of $G_{F}$ or $\pi_{1}^{\top}(X)$ for some model of $X$, we will denote by $T_{v}^{a}$ and $D_{v}^{a}$ the images of inertia and decomposition, respectively, in $\Pi^{\ab}$, which will sometimes appear in the sequel as $H_{1}$.  We let $t_{v}^{a}$ \index{notation}{tva@$t_{v}^{a}$} be the image of a meridian in $\Pi^{\ab}$.
\end{definition} 

We can also define the meridian of a valuation $v$ on a model $X$ if $|v|$ is smooth, and extend the definition to non-smooth $|v|$ as follows.  We resolve the singularities of $|v|$ on $X$ to get a birational map
\[
\eta: \tilde{X}\rightarrow X
\]
such that
\begin{enumerate}
\item $\eta$ is an isomorphism outside of $|v|$.
\item $|v|\subseteq \tilde{X}$ is smooth.
\end{enumerate}
Then we define a meridian $t_{v}$ on $X$ to be
\begin{equation}\label{def:meridian2}
t_{v} = \eta_{*}(t_{v})
\end{equation}
where $t_{v}$ is a meridian on $\tilde{X}$.  To see this is well-defined, if we have two such maps $\eta, \eta'$ as in the following diagram:
\begin{equation}
\mbox{\xymatrix{
& \tilde{X}'' \ar@{-->}[dl]_{\varphi} \ar@{-->}[dr]^{\varphi'} & \\
\tilde{X} \ar[dr]^{\eta} & & \tilde{X}' \ar[dl]_{\eta'} \\
& X, & }}
\end{equation}
we may always construct $\varphi$ and $\varphi'$ birational morphisms so that:
\begin{enumerate}
\item The above diagram commutes.
\item $\varphi, \varphi', \eta \circ \varphi, \eta'\circ \varphi'$ are isomorphisms outside of $|v|$.
\item $|v|$ is smooth in $\tilde{X}''$.
\end{enumerate}
In this case, the meridians in $X$ defined by $\eta'$ and $\eta$ are the image of a meridian in $\tilde{X}''$ under $\eta \circ \varphi$ and $\eta' \circ \varphi'$ so, by commutativity of the diagram, the two meridians are the same.

There is also the notion of a meridian for a higher-rank Parshin chain; we give here the notion for a rank-$2$ Parshin chain.
\begin{definition}\label{def:meridianrank2} Let $L$ be an algebraic extension of $F$, $w\circ v$ a rank-$2$ Parshin chain on $F$, and $X$ a model of $F$ on which $v$ has a center, but on which $w\circ v$ is not centered.  Then the \textbf{meridian of the rank-$2$ Parshin chain} $t_{w\circ v}$\index{notation}{twv@$t_{w\circ v}$} for $L|F$ is the element of $\Aut(L|F)$ induced by the inverse limit of the monodromies of a loop on $|v|$ around the point induced by $w$ on the normalization of $|v|$.  As in \autoref{not:abelianization}, any image in an abelianization will be denoted $t_{w\circ v}^{a}$\index{notation}{twva@$t_{w\circ v}^{a}$}.
\end{definition}
Let $F$ have dimension $n$.  Then if $v$ is an $n-1$-dimensional Parshin chain, $Fv$ is the function field of a curve over $K$.  $Fv$ is equipped with a fundamental, birational invariant: its \textbf{unramified genus} $g(v)$\index{definitions}{Unramified genus}\index{notation}{gv@$g(v)$}.  We can compute this as follows:
\[
g(v) = \rk_{\hat{\mathbb{Z}}}Dv^{a}/\langle Tv^{a}, T_{p}^{a}\rangle_{p \in \Par_{n}(v)}.
\]
\section{Geometric Sets and the Maximal Smooth Model}\label{models}
\begin{definition}\label{def:geometricset} \index{definitions}{Geometric set}We say that a set $\caS$ of prime divisors of a function field $F$ is a \textbf{geometric set (of $F$)} if and only if there exists a normal model $X$ of $F$ such that $\caS$ is precisely the set of valuations with centers Weil prime divisors on $X$.  In this case, we write 
\[
\caS = \caD(X).
\]
\index{definitions}{Model}If $X$ is smooth, we say $X$ is a \textbf{model} of $\caS$.
\end{definition}
\begin{theorem}[Pop]\label{ValuationsRecipe}
If $F$ is a function field with $K(F) = \overline{\mathbb{Q}}$, $\trdeg_{\overline{\mathbb{Q}}} F \geq 2$, and let $\Gamma \subseteq G_{F}$ be a closed subgroup, up to conjugacy.  Then there is a topological group-theoretic criterion, given one of the representatives of $\Gamma$ to determine whether there exists $i$ and $v \in \Par_{i}(F)$ such that $\Gamma = T_{v}$ or $\Gamma = D_{v}$, and what this $i$ is if it exists.
\end{theorem}
This theorem is proven with $G_{F}$ replaced by the pro-$\ell$ completion of $G_{F}$ in \cite{PopBog2}.  To see this for $G_{F}$ as a whole, we may apply Key Lemma 5.1 of \cite{PopIOM}.  As the maximal length of a Parshin chain is the transcendence degree of $F$, this recipe immediately determines the transcendence degree of $F$.  Pop also proved \cite{PopIOM}:
\begin{theorem}[Pop]\label{theorem:GeometricSetsRecipe}
Given a geometric set $\caS$ of prime divisors on $F$,
\begin{enumerate}
\item If $\caS$ is a geometric set of prime divisors on $F$, then a (possibly different) set $\caS'$ of prime divisors on $F$ is a geometric set if and only if it has finite symmetric difference with $\caS$.
\item There exists a group-theoretic recipe to recover\index{notation}{GeomF@$\caGeom(F)$}
\[
\caGeom(F) = \left\{ \left\{(T_v, D_v) \mid v \in \caS\right\} \mid \caS\text{ a geometric set}\right\}.
\]
\end{enumerate}
\end{theorem}
\begin{definition}\label{def:geometricsetfundamentalgroup}\index{definitions}{Fundamental group of a geometric set}
If $\caS$ is any set of prime divisors on $F$, we define the \textbf{fundamental group of }$\caS$ to be:\index{notation}{PiS@$\Pi_{\caS}$}
\[
\Pi_{\caS} = G_{F}/\langle T_{v}\rangle_{v\in \caS}.
\]
Here, $\langle T_{v}\rangle_{v\in \caS}$ is the smallest closed, normal subgroup of $G_{F}$ which contains every element in every conjugacy class in each $T_{v}$.  If $\caT\subseteq \caS$ is a subset, then we denote by\index{notation}{rhoTS@$\rho_{\caT\caS}$}
\[
\rho_{\caT\caS}: \Pi_{\caT}\rightarrow \Pi_{\caS}
\]
the restriction map, and drop subscripts when they are unambiguous.
\end{definition}
Given a geometric set $\caS$, there are many possible $X$ such that $\caD(X) = \caS$; for instance, any model less a finite subset of points has the same set of prime divisors. We now define the maximal model on which we will be able to effect our intersection theory.
\begin{theorem}\label{thm:maximalsmoothmodels}
\index{definitions}{Model!Maximal smooth}Let $\caS$ be a geometric set for a function field $F$ of dimension $2$.  There exists a \textit{unique} model $\caM(\caS)$ \index{notation}{MS@$\caM(\caS)$}of $F$ such that the following holds:
\begin{enumerate}
\item $\caD(\caM(\caS)) = \caS$.
\item $\caM(\caS)$ is smooth.
\item \label{equivalenceoffundamentalgroups}$\hat{\pi}_{1}(\caM^{\anal}, p) \simeq \poe(\caM(\caS), p) \simeq \Pi_{\caS}$.
\item \label{maximalmodeluniversalproperty} If $X$ is any other smooth model of $F$ which satisfies $\caS = \caD(X)$, then there exists a unique $F$-morphism $X\rightarrow \caM(\caS)$, and this is a smooth embedding.
\end{enumerate}
\end{theorem}
\begin{proof}
Let $U$ be a model of $\caS$, and let $X$ be a smooth compactification of $U$.  Let \[\partial = \caD(X)\setminus \caD(U)\]
be the collection of field-theoretic prime divisors in the boundary of $U$ in $X$.  This is a finite set, as the boundary divisor is itself a finite union of prime divisors.  We now define a sequence of pairs $(X_{i}, \partial_{i})$ of varieties $X_{i}$ and finite sets of divisors $\partial_i \subset \caD(X)$ inductively as follows:
\begin{enumerate}
\item Let $X_{1} = X, \partial_{1} = \partial$.
\item We now construct $(X_{i+1}, \partial_{i+1})$ from $(X_i, \partial_i)$.  First, take the collection $\{v_j\} \subseteq \partial_i$ such that each $|v_j|$ is a $-1$-curve such that no other $|v'|$ that intersects it in the boundary is a $-1$-curve, and blow down. Set $X_{i+1}$ to be this blowdown, and $\partial_{i+1} = \partial_{i}\setminus \{v_j\}$.
\end{enumerate}
As $\partial_{1}$ is finite, at some point, this sequence becomes stationary --- let's say at $(X_n, \partial_{n})$.  Then we define \begin{equation}\label{not:Umax} U_{\max} = X_{n} \setminus \bigcup_{v \in \partial_{n}} |v|.\end{equation}  To prove that this satisfies property~\ref{maximalmodeluniversalproperty}, let $U'$ be another model and $X'$ a smooth compactification of it.  Run the algorithm on $X'$ to get a pair $(X'_{n'}, \partial'_{n'})$ Then by strong factorization for surfaces (see Corollary 1-8-4, \cite{MoriProgram}), there exists a roof
\begin{center}
\mbox{\xymatrix{
& Y \ar[dl]_{\delta} \ar[dr]^{\delta'} & \\
X'_{n'} & & X_{n},
}}
\end{center}
where $\delta$ and $\delta'$ are both sequences of blowups.  For any morphism \[
\varphi: Z\rightarrow Z'
\]
 of varieties we may define the \textbf{exceptional locus}
\[
\caE(\varphi) = \left\{p \in Z' \mid \dim(\varphi^{-1}(p)) \geq 1\right\}
\]
Let $p \in \caE(\delta')\cap U_{\max}$.  Then $\delta'^{-1}(p)$ is connected, so $\delta(\delta'^{-1}(p))$ is also connected.  It is proper, so is either a union of divisors or a point. If it is a union of divisors and one of these divisors were in $U'_{\max}$, this divisor would be contracted, so $U_{\max}$ and $U'_{\max}$ would not have the same codimension $1$ theory.  We may argue the same way for $\delta$.  Thus, $\delta\circ \delta'^{-1}|_{U_{\max} \setminus \caE(\delta')}$ is well-defined and regular outside codimension $2$, so extends to a morphism \[\delta\circ\delta'^{-1}: U_{\max}\rightarrow U'_{\max},\] injective on closed points.  By the same argument we may produce the inverse \[\delta'\circ \delta^{-1}: U'_{\max}\rightarrow U_{\max},\] so we have that the maximal smooth model is indeed unique up to isomorphism.  

To prove property \ref{equivalenceoffundamentalgroups}, we note that the map
\[
U\rightarrow U_{\max}
\]
gives a natural equivalence of the category of \'etale covers of $U$ with the category of \'etale covers of $U_{\max}$, by the Nagata-Zariski purity theorem \cite[X.3.4]{SGA2}, and thus gives an isomorphism on fundamental groups by \autoref{thm:fundamentalgroupscomparison}.

\end{proof}
\begin{corollary}
Let $U$ be an affine or projective smooth variety with function field $F$.  Then $U = \caM(\caD(U))$.
\end{corollary}
\begin{proof}
If $U$ is proper, the algorithm in the proof of \autoref{thm:maximalsmoothmodels} terminates immediately, so $U_{\max} = U$.

If $U$ is affine, let $\iota: U\rightarrow U_{\max}$ be the embedding of $U$ as an affine open of $U_{\max}$.  Assume there were a closed point $x\in U_{\max}\setminus U$.  Then there is an affine neighborhood $U'\subset U_{\max}$ such that $x \in U'$.  Then $U'\cap U$ is affine.  Its complement $U' \setminus (U'\cap U)$ must then contain a divisor, which contains $x$, and is not contained in $U$.  Thus, $U_{\max}$ has a different codimension-$1$ theory from $U$, which gives a contradiction.
\end{proof}

Now, let $F$ be a function field over $\overline{\mathbb{Q}}$, and fix an embedding of $\overline{\mathbb{Q}}$ into $\mathbb{C}$.  Then if $\caX_{F}$ is the inverse system of all smooth models of $F$,
\begin{equation}\label{GGInvLimofPi1}
G_{F} \simeq \lim_{X \in \caX_{F}} \hat{\pi}_{1}((X\times_{\overline{\mathbb{Q}}}\Spec\mathbb{C})^{\anal}), 
\end{equation}
where we will leave the notion of basepoint ambiguous (as we never need to specify it); the isomorphism \ref{GGInvLimofPi1} is highly noncanonical, but well-defined up to conjugation.
\section{The Local Theory: The Intersection Theorem}
The following theorem shows how the fundamental group detects intersections in the best-case scenario.
\begin{theorem}[The Local Anabelian Intersection Theorem]\label{thm:localanabelianintersectiontheorem}
Let $X$ be a smooth (not necessarily proper) surface over $\mathbb{C}$, and let $C_1$ be a unibranch germ of an algebraic curve at a point $p\in X$ and $C_2$ an irreducible, reduced algebraic curve on $X$ (a prime divisor), with distinct branches $\gamma_j$ at $p$, with $C_1$ distinct from each branch of $C_2$.  Let $Y = X \setminus C_2$ be the complement of $C_2$ in $X$; this is an open subvariety of $X$.  Let $w\circ v_{1}$ be the rank-two Parshin chain corresponding to $C_1$ at $p$, and let $v_{2}$ be the Weil prime divisor corresponding to $C_2$.  Then, we may choose meridians $t_{v_{2}}^{a}$ and $t_{w\circ v_{1}}^{a}$ in $\poe(X\setminus C_{2})^{\ab}$ 
\begin{equation}
t_{w\circ v_{1}}^{a} = i(p, C_1\cdot C_2; X)\,t_{v_{2}}^{a},
\end{equation}
where $i(p, C_{1}\cdot C_{2}; X)$\index{definitions}{Local intersection number}\index{notation}{intp@$i(p, C_{1}\cdot C_{2}; X)$} is the local intersection number as defined in Fulton \cite[Definition 7.1]{IntersectionTheory}. Then 
\begin{equation}
[T_{v_{2}}(F_{X\setminus C_2}|F)^{a}: T_{w\circ v_{1}}(F_{X\setminus C_2}|F)^{a}] | i(p, C_1\cdot C_2; X)
\end{equation}
with equality iff $i(p, C_1\cdot C_2; X)|\left|T_{v_{2}}(F_{X\setminus C_2}|F)^{a}\right|$ or $T_{v_{2}}(F_{X\setminus C_2}|F)^{a}$ or $T_{w\circ v_{1}}(F_{X\setminus C_2}|F)^{a}$ are infinite in $\poe(X\setminus C_2)^{\ab}.$
\end{theorem}
\begin{lemma}[Reeve's Lemma]\label{thm:reeveslemma}
We use the notation of \autoref{thm:localanabelianintersectiontheorem}.  There is a neighborhood $B\subseteq X$ of $p$ biholomorphic to a unit ball $B \subseteq \mathbb{C}^{2}$ about the origin so that $p$ corresponds to $(0, 0)$.  Let $\partial B$ be the boundary of $B$, which is homeomorphic to a $3$-sphere.  Then for small enough such $B$,
\begin{enumerate}
\item $H_{1}(\partial B\setminus (C_{2} \cap \partial B), \mathbb{Z})\simeq \bigoplus_{j} \mathbb{Z}t_{\gamma_{j}}^{a}$, the direct sum generated by the meridian around each branch, where each curve goes in the counterclockwise direction after identifying the normal bundle with an open disc in $\mathbb{C}$.
\item $C_{1}$ intersects $\partial B$ in a loop $\ell$, which can be given an orientation by the complex orientation on $X$.  Then in $H_{1}(\partial B\setminus (C_{2}\cap \partial B), \mathbb{Z})$, we have
\begin{equation}
\ell = \sum_{j} i(p, C_{1}\cdot \gamma_{j}; X) t_{\gamma_{j}}^{a}.
\end{equation}
\end{enumerate}
\end{lemma}
We include a simple proof of this lemma (communicated to us by David Massey), which seems to have been first written down in the difficult-to-access Reeve \cite{Reeve}, and it seems to be sufficiently well-known to experts that it does not occur elsewhere in the literature.
\begin{thesisonly}
\newpage
\end{thesisonly}
\begin{proof}
We prove the two assertions in order.
\begin{enumerate}
\item Embedded algebraic curves have isolated singularities, so each $\gamma_{j}$ intersects $\partial B$ in a smoothly embedded $S^{1}$ (by the inverse function theorem), and each of these circles are disjoint.  The Mayer-Vietoris sequence gives the first claim immediately.
\item Let \begin{equation}f_{j}(x,y) = 0\end{equation} be a local equation for $\gamma_{j}$, and let \begin{equation}
u\mapsto (x(u), y(u))\end{equation} be a local parameterization of the normalization $\tilde{C}_{1}$ of $C_{1}$ at $p$.  By \cite[Example 1.2.5b]{IntersectionTheory}
\begin{equation}
i(p, C_{1}\cdot \gamma_{j}; X) = v_u(f_{j}(x(u), y(u)),
\end{equation}
the $u$-adic valuation --- or order of vanishing --- of $f_{j}(x(u), y(u)) \in \mathbb{C}[[u]]$.

Then the intersection with $\partial B$ of $C_{1}$ can be taken to be a path $\eta: [0,1]\longrightarrow \mathbb{C}$ with winding number one around the origin so that
\begin{equation}
\ell(t) = (x(\eta(t)), y(\eta(t))).
\end{equation}
Let
\begin{equation}
\pi_{j}: H_{1}(\partial B \setminus (C_{2}\cap \partial B), \mathbb{Z})\longrightarrow \mathbb{Z} t_{\gamma_{j}}^{a}
\end{equation}
be the projection given by the direct sum decomposition as above.  Consider the differential $1$-form $d\log f_{j}(x,y)$ restricted to $\partial B$.  Then
\[
\pi_{j}(\ell(t)) = \frac{1}{2\pi i}\int_{\ell(t)} d\log f_{j}(x,y),
\]
on $B$, as $d\log f_{j}(x,y)$ is holomorphic away from the zero-set of $f_{j}$ and otherwise measures the winding of a loop around the zero-set.  But this pulls back to the contour integral
\[
\frac{1}{2\pi i}\int_{\eta(t)} d\log f(x(u(t)), y(u(t)))
\]
which evaluates exactly to $v_{u}(f_{j}(x(u), y(u)))$ by the Residue Theorem.
\end{enumerate}
\end{proof}

\begin{proof}[Proof of \autoref{thm:localanabelianintersectiontheorem}]
Let $\partial B$ be as in \autoref{thm:reeveslemma}.  Then given the map
\begin{equation}
\iota: \partial B\setminus (C_{2}\cap \partial B)\longrightarrow X,
\end{equation}
we must compute
\begin{equation}
\iota_{*}: H_{1}( \partial B\setminus (C_{2}\cap \partial B), \mathbb{Z}) \longrightarrow H_{1}(X, \mathbb{Z}).
\end{equation}
But from \autoref{thm:geometrictheoryofdecompositionandinertia}, for each $j,$
\[
i_*(t_{\gamma_j}^{a}) = t_{v_{2}}^{a}
\]
and in the notation of \autoref{thm:geometrictheoryofdecompositionandinertia},
\[
\ell = t_{w\circ v_{1}}^{a}.
\]
\begin{thesisonly}
\newpage
\end{thesisonly}
Thus by linearity, we have
\[
t_{w\circ v_{1}}^{a} = \sum_j i(p, C_1\cdot \gamma_j; X)\,t_{v_{2}}^{a} = i(p, C_1 \cdot C_2; X)\,t_{v_{2}}^{a}.
\]
\end{proof}
\section{The Global Theory I: Points and Local Intersection Numbers}\label{globaltheoryI}
\begin{definition}\label{points} Let $\caS$ be a geometric set.
\begin{enumerate}
\item Let $v \in \caS$.  We define for every rank $2$ Parshin chain $p\circ v$ the subset \index{notation}{Delta@$\Delta$} 
\begin{equation}\label{eqn:Delta}
\Delta(p\circ v) = \left\{w\in \caS \left| \begin{array}{l} \forall\,Y\subset \caS\text{ finite, s.t. }w \in Y\text{ and }T_{w}\text{ is torsion-free in } \\
\Pi_{\caS\setminus Y}^{\ab}, T_{p\circ v}^{a} \neq\{0\}\text{ in }\Pi^{\ab}_{\caS\setminus Y}\end{array}\right.\right\}.
\end{equation}

\item\label{def:point} We say $p\circ v\sim p'\circ v'$ if and only if $\Delta(p\circ v) = \Delta(p'\circ v')$, and the equivalence class of rank-two Parshin chains will be denoted by $[p\circ v]$ and called a \textbf{point}\index{definitions}{Point}.  Given such a point $[p\circ v]$, we will use $\Delta([p\circ v])$ to denote, for any $p' \circ w \in [p\circ v]$, $\Delta(p'\circ w)$.
\item\label{def:intersects} \index{definitions}{Intersection}We say that a prime divisor $w\in \caS$ \textbf{intersects} a point $[p\circ v]$ if $w \in \Delta([p\circ v])$.
\item\label{not:setofpoints} The set of points on $\caS$ is denoted $\caP(\caS)$.\index{notation}{P@$\caP$}
\item\label{not:Pofv} If $v \in \caS$ then $\caP(v)$ will be the set of points which contain an element of the form $p\circ v$.
\item\label{not:PofModel} We denote by $\caM(\caS)$ the closed $K$-points of the maximal smooth model $\caM(\caS)$ of $\caS$.
\end{enumerate}
\end{definition}
By \autoref{thm:localanabelianintersectiontheorem}, \autoref{eqn:Delta} roughly says that whenever we have removed so many divisors that $T_{w}^{a}$ is large enough to detect any possible intersection between $|p\circ v|$ and $|w|$, it does.
\begin{definition}\label{def:recognizes} \index{definitions}{Recognition}Let $Y \subset \caS$ be a finite subset.  Then we say that $Y$ \textbf{recognizes} the intersection of $v$ at $p$ if $T_{v}^{a}$ is torsion-free in $\Pi_{\caS \setminus Y}^{\ab}$ and $Y \cap \Delta(p) = \{v\}$.
\end{definition}
\begin{definition}\label{def:totalintersectionproduct}\index{definitions}{Total intersection product}\index{notation}{intD1D2@$(D_{1}\cdot D_{2})$}
Let $X$ be a smooth surface over an algebraically closed field of characteristic zero. Let $D_{1}, D_{2}$ be two divisors on $X$.  We define the \textbf{total intersection product}
\begin{equation}
(D_{1}\cdot D_{2}) = \sum_{p \in D_{1}} i(p, D_{1}\cdot D_{2}; X).
\end{equation}
\end{definition}
\begin{theorem}[The Algebraic Inertia Theorem]\label{thm:algebraicinertia}
In this theorem (and its proof), all homology will be taken with integral coefficients.  Let $\{D_{i}\}_{i\in I}$ be a finite collection of smooth, distinct prime divisors on a smooth, proper, complex, algebraic surface $X$ such that \begin{equation}D = \bigcup_{i\in I} D_{i}\end{equation} is simple normal crossing.  Let
\begin{equation}
\eta_{1}: \NS(X)\rightarrow \bigoplus_{i\in I} \mathbb{Z}t^{a}_{v_{i}}
\end{equation}
be given by
\begin{equation}\label{def:etaone}
\eta_{1}(D) = \bigoplus_{i \in I} (D\cdot |v_{i}|)t^{a}_{v_{i}}.
\end{equation} Then there is a short-exact sequence
\begin{equation}\label{diag:algebraicinertiasequence}
\mbox{\xymatrix{
\NS(X) \ar[r]^-{\eta_{1}} & \bigoplus_{i\in I}\mathbb{Z}t^{a}_{v_{i}} \ar[r] &  \langle t_{v_{i}}^{a}\rangle_{i\in I} \subseteq H_{1}(X\setminus D) \ar[r] & 0. }}
\end{equation}
\end{theorem}
We omit the proof, as it is a long, but straightforward application of the Mayer-Vietoris sequence and the Hodge-Lefschetz $(1,1)$-theorem.
\begin{corollary}[The Separation of Inertia Criterion]\label{cor:separationofinertia}We preserve the notation of \autoref{thm:algebraicinertia}.  Suppose that, in addition to the hypotheses, for each $i, j \in I$ there does not exist $E \in \NS(X)$ s.t.~for every $k \in I \setminus \{i,j\}$,
\begin{equation}\label{alginertiatheoremintersectioncriterion}
(E\cdot D_{k}) = 0,\text{ but }(E\cdot D_{i}) \neq 0 \text{ or }(E\cdot D_{j}) \neq 0.
\end{equation}
For each $i\in I$, let $v_{i}$ be the prime divisor on $\mathbb{C}(X)$ associated to $D_{i}$.  Then in $H_1(X \setminus D)$,
\begin{enumerate}
\item $T_{v_{i}}^{a}\simeq \mathbb{Z}$.
\item For $j\in I, j \neq i,$
\begin{equation}
T_{v_{i}}^{a}\cap T_{v_{j}}^{a} = 0.
\end{equation}
\end{enumerate}
\end{corollary}
\begin{proof}
We apply the hypothesis in \autoref{cor:separationofinertia} to \autoref{diag:algebraicinertiasequence} to see that there are no elements in the image of $\partial$ none which are zero at all but one component to prove the first claim, and which are zero at all but two components to prove the second claim.
\end{proof}
\begin{lemma}[Divisor Existence Lemma]\label{lem:divisorexistence}
Let $D_{1}, \dots, D_{j}$ be a finite set of divisors on a smooth surface, and $p_{1}, \dots, p_{n}$ be a finite set of points.  Then there are infinitely (in fact ``generically'') many prime divisors which intersect each of the $D_{i}$ but not at the $p_{i}$.
\end{lemma}
\begin{proof}
Choose a very ample divisor $C$.  Then $C\cdot D_{i}> 0$ for all $i$.  Then having an intersection at $p_{i}$ is a closed condition (since $C$ is basepoint-free), and so the divisors which do not intersect at those points form an open, nonempty subset of the linear system $|C|$, which is then infinite.
\end{proof}
\begin{corollary}\label{cor:pointsarepoints}
Let $\caS$ be a geometric set and let $P \in \caM(\caS)$.  Let $v \in \caS$ and let $p \circ v$ be a rank-$2$ Parshin chain such that \begin{equation} |p\circ v| = P.\end{equation}  Then
\begin{enumerate}
\item $w \in \Delta(p\circ v)$ if and only if $P \in |w|.$
\item If $w \in \Delta(p\circ v)$ then there exists a set $Y\subset \caS$ which recognizes the intersection of $w$ and $[p\circ v]$.
\item If $v' \in \caS$ and $p' \circ v'$ is not centered on $\caM(\caS)$, then $\Delta(p'\circ v') \neq \Delta(p \circ v)$.
\end{enumerate}
Then there is a canonical map
\begin{equation}\label{not:iotaonpoints}\index{notation}{iota@$\iota$}
\iota: \caM(\caS)\longrightarrow \caP(\caS),
\end{equation}
given by
\begin{equation}
\iota(x) = \left\{v \in \caS \mid x \in |v|\right\}.
\end{equation}
This map is always an injection, and is a bijection if and only if $\caM(\caS)$ is proper.
\end{corollary}
\begin{proof}
Let $X$ be a smooth compactification of $\caM(\caS)$ with only simple normal crossings at the boundary and let \begin{equation}\partial \caS = \caDiv(X) \setminus \caM(\caS) \subset  \Par_{1}(F)\end{equation} be the prime divisors supported as prime divisors on the boundary of $X$.  
\begin{enumerate}
\item Let $P \in |w|$ and $\Upsilon \subset \caS$ a cofinite subset which does not contain $w$ and so that $T_{w}^{a}$ is torsion-free in $\Pi_{\Upsilon}^{\ab}$.  Then \autoref{thm:reeveslemma} implies that the projection of $T_{p\circ v}^{a}$ to the $T_{w}^{a}$-part of the direct summand is torsion-free in $T_{w}^{a}$ in any $\Pi_{\Upsilon}^{\ab}$, so must be torsion-free itself, so $w$ satisfies the properties of an element of $\Delta(p\circ v)$.
\item \autoref{lem:divisorexistence} shows that there exists a finite set $\{\beta_{i}\}_{i \in I}$ of prime divisors such that $\{|\beta_{i}|\}_{i\in I}\cup \{|w|\} \cup \{|\beta|\}_{\beta \in \partial \caS}$ satisfies the hypotheses of \autoref{cor:separationofinertia}, and so $\{\beta_{i}\}_{i \in I}\cup \{w\}$ recognizes the intersection of $w$ with $p\circ v$.
\item This follows immediately from the fact that there exists a prime divisor in $\caM(\caS)$ which intersects the boundary at $|p' \circ v'|$ and does not intersect $|p\circ v|$, so the valuation associated to this divisor will be in $\Delta(p'\circ v')$ but not in $\Delta(p\circ v)$.
\end{enumerate}
The existence and injectivity of the map $\iota$ is now straightforward.  The bijectivity in case of properness follows from the valuative criterion for properness.
\end{proof}
\begin{definition}\label{def:localintersectionnumber}
Let $\caS$ be a geometric set, $v, w \in \caS$, and $p\circ w$ a rank-$2$ Parshin chain.  Then we define the \textbf{local intersection number} $(p, v\cdot w; \caS)$ \index{definitions}{Local intersection number}\index{notation}{pvw@$(p, v\cdot w; \caS)$} as follows:
\begin{enumerate}
\item If there does not exist a set $Y$ which recognizes the intersection of $v$ at $p$ (and this includes the case where $v \notin \Delta(p)$, then we define
\begin{equation}
(p, v\cdot w; \caS) = 0.
\end{equation}
\item Otherwise, let $Y$ recognize the intersection of $v$ at $p$. Then we define \[
(p, v\cdot w; \caS) = [T_{v}: T_{p\circ w}]_{\Pi_{\caS \setminus Y}}.
\]
\end{enumerate}
\end{definition}
\begin{theorem}\label{thm:intersectionproduct}
Let $\caS$ be a geometric set, $p \in \caP(\caS)$ a point with a center $|p| \in \caM(\caS)$ and $v\in \caS$.  Then
\begin{equation}\label{formula:intersectionproduct}
\sum_{p'\circ w\in p, w \in \caS}(p', v\cdot w; \caS) = i(|p|, |v|\cdot |w|; \caM(\caS)).
\end{equation}
\end{theorem}
\begin{proof}By \autoref{cor:pointsarepoints}, either both sides of \autoref{formula:intersectionproduct} are zero, or there exists a geometric set which recognizes the intersection of $p$ with $v$.  Now, each $p'\circ w \in p$ represents a branch of $|w|$ at $p$, and we will call this germ $\xi_{p'\circ w}$.  We then have by the Local Anabelian Intersection Formula
\begin{equation}
i(|p|, |v|\cdot |w|; \caM(\caS))= \sum_{|p'\circ w| \in p} i(|p|, |v|\cdot \xi_{p'\circ w}; \caM(\caS)) = \sum_{|p'\circ w| \in p} (p', v\cdot w; \caS).
\end{equation}
\end{proof}

\begin{definition}
Let $p\circ v$ be a rank-$2$ Parshin chain with $v \in \caS$.  Then we say that $p\circ v$ is a \textbf{nonnodal chain}\index{definitions}{Nonnodal chain} for $\caS$ if $[p\circ v]$ is distinct from every other $[p'\circ v]$, and we say that $p\circ v$ is a \textbf{noncuspidal chain}\index{definitions}{Noncuspidal chain} if there exists a $v' \in \caS$ with $(p, v'\cdot v; \caS) = 1$.  A rank-$2$ Parshin chain which is both nonnodal and noncuspidal will be called a \textbf{smooth chain}\index{definitions}{Smooth chain}.
\end{definition}
\begin{proposition}
If $|p\circ v| \in \caM(\caS)$ then $p\circ v$ is nonnodal (resp.~noncuspidal) if and only if $|v|$ is not nodal (resp.~cuspidal) at $|p\circ v|$.
\end{proposition}
\begin{proof} This follows directly from \autoref{cor:pointsarepoints} and the definition of a nodal, resp.~cuspidal, point on a curve.
\end{proof}
Recall that $\caP(\caS) \subset 2^{\Par_{2}(F)}$; thus, each point of each geometric set is a subset of a larger set, and we will consider them as sets in the next sequence of definitions.
\begin{definition}\
Let $\caS'$ be a geometric set, and let $p \in \caP(\caS')$.  Then we define the \textbf{$\caS'$-limits of $p$ in $\caS$} by\index{definitions}{Limits|(}\index{notation}{Lim@$\Lim$|(}
\begin{equation}
\Lim_{\caS}(p) = \{\pi \in \caP(\caS) \mid \pi \cap p \neq \emptyset\}. 
\end{equation}
\end{definition}
\begin{definition}
If $v \in \caS'$, then we define the \textbf{$\caS'$-limits of $v$ in $\caS$} to be
\begin{equation}
\Lim_{\caS}(v) = \bigcup_{p \circ v \in \Par_{2}(v)} \Lim_{\caS}([p\circ v]).
\end{equation}
\index{definitions}{Limits|)}\index{notation}{Lim@$\Lim$|)}
\end{definition}
\begin{definition}
Let $\caS'$ be a geometric set so that $\caS\subset \caS'$.  Then we define the \textbf{boundary points of $\caS$ relative to $\caS'$} \index{definitions}{Boundary points} \index{notation}{bSS@$\partial_{\caS'}(\caS)$} as
\begin{equation}
\partial_{\caS'}(\caS) = \bigcup_{v \in \caS' \setminus \caS} \Lim_{\caS}(v),
\end{equation}
and the \textbf{interior points of $\caS$ relative to $\caS'$}\index{definitions}{Interior points}\index{notation}{PSS@$\caP_{\caS'}(\caS)$} as
\begin{equation}
\caP_{\caS'}(\caS) = \caP(\caS) \setminus \partial_{\caS'}(\caS).
\end{equation}
\end{definition}
\begin{definition}
We say that a point $[p\circ v]$ is \textbf{absolutely uncentered}\index{definitions}{Absolutely uncentered point} on $\caS$ if 
$T_{p'\circ v'}$ is nontrivial in the total fundamental group $\Pi_{\caS}$
of $S$ for some $p'\circ v' \in [p\circ v]$. We define\index{notation}{aS@$a(\caS)$}
\[
a(\caS) = \left\{[p\circ w] \in \caP(\caS) \mid [p\circ w]\text{ absolutely uncentered}\right\}.
\]
We define the \textbf{candidate points of $\caS$} \index{definitions}{Candidate points}\index{notation}{AS@$\caA(\caS)$}by \[
\caA(\caS) = \caP(\caS) \setminus a(\caS),
\]
and these are the points which are not absolutely uncentered.
\end{definition}
We have immediately:
\begin{proposition}
For any geometric set $\caS$, \[\caM(\caS) \subseteq \caA(\caS).\]  That is, absolutely uncentered points of $\caS$ do not have centers on $\caM(\caS)$, and candidate points have a chance.
\end{proposition}
The converse to this proposition is false in general:
\begin{example}
Let $F = \overline{\mathbb{Q}}(x,y)$, and $X = \Spec \overline{\mathbb{Q}}[x,y]$ and $\caS = \caDiv(X)$.  If $p\circ v$ is not centered on $X$, then the algebraic inertia theorem (accounting for resolution of singularities) shows that $T_{p\circ v}^{a} = \mathbb{Z}$ in any divisor complement, so \begin{equation}\Delta(p\circ v) = \caS.\end{equation}  But as $\Pi_{\caS}$ is trivial, \begin{equation}\caA(\caS) = \caP(\caS).\end{equation} Thus,
\begin{equation}
\caA(\caS) = X \cup \{\infty\}.
\end{equation}
\end{example}
However, for the ``visible affine opens'' we define below, the converse is true.  The first goal of Anabelian Intersection Theory is to identify these special geometric sets, and use this to construct a salvage of the converse.
\section{The Global Theory II: Visible Affines and Properness}
In this section, we fix a two-dimensional function field $F$.
\begin{definition}\label{def:visaff}
Let $U$ be a model of $F$ which admits a surjective map
\[
\pi: U\longrightarrow B
\]
to a hyperbolic curve $B$, with smooth, hyperbolic fibers of the same genus with at least three punctures.  We call $U$ a \textbf{visible affine of $F$}\index{definitions}{Visible affine}  (this is a topological fibration, if not a Zariski fibration).  There is then a \textbf{horizontal-vertical decomposition}\index{definitions}{Horizontal-vertical Decomposition}\index{notation}{H@$\caH$}\index{notation}{V@$\caV$}
\[
\caD(U) = \caH \cup \caV
\]
into \textbf{horizontal} divisors (the members of $\caH$) and \textbf{vertical} divisors (the members of $\caV$), where the vertical divisors are given as the fibers of $\pi$.  This $\pi$ determines the horizontal-vertical decomposition, and  a horizontal-vertical decomposition determines $\pi$ up to automorphisms of the base.
\end{definition}
\begin{proposition}\label{prop:geomsetsproperties}
Let $U$ be a visible affine of $F$ with horizontal-vertical decomposition \begin{equation}\caD(U) = \caH\cup \caV.\end{equation}  Then \begin{enumerate}
\item $\iota(\caP(U)) = \caA(\caD(U)).$
\item For any $v, v' \in \caV, D_{v} = D_{v'}$ in $\Pi_{\caD(U)}$.
\item For any $v \in \caV$, $\poe(B) = \Pi_{\caD(U)}/D_{v}$.
\item Let \begin{equation}\partial B = \overline{B}\setminus B.\end{equation}  Then for each $h\in \caH$ and $p \in \partial B$ there exists $q\circ h \in \Par_{2}(h)$ such that $0 \neq \pi_{*}(T^{a}_{q\circ h}) \subseteq T^{a}_{p}$, the closure of a group generated by a meridian around $p$ in $B$.  In particular, let 
\[
I_{U}^{h} = \langle T^{a}_{p}\rangle_{p \in \Par_{2}(h)}
\]
 be the closure of the subgroup of $\poe(U)^{\ab}$ generated by all inertia of rank-$2$ valuations, and let $I_{B}$ be the divisible hull of $\pi_{*}(I_{U}^{h})\subseteq \poe(B)^{\ab}$.  Then \index{notation}{IB@$I_{B}$}$I_{B}$ is independent of the choice of $h$, and
 \begin{equation}\index{notation}{gB@$g(B)$}
g(B) = \rk_{\hat{\mathbb{Z}}}(\poe(B)^{\ab}/I_{B}),
\end{equation}
where $g(B)$ is the unramified genus of $B$.  
\end{enumerate}
\end{proposition}
\begin{proof}
\begin{enumerate}
\item By \autoref{cor:pointsarepoints}, $\iota(U) \subseteq \caA(\caD(U))$. Let $C$ be any smooth, hyperbolic, possibly open curve.  Then for any choice of basepoint $p \in C, \pi_{2}(C, p) = 0$ and $\pi_{1}^{\top}(C, p)$ is residually finite.  Thus, for $v \in \caV$ and $p \in |v|$ there is a short-exact fiber sequence
\begin{equation}\label{ses:fibersequenceforvisibleaffine}
1\longrightarrow \poe(|v|, p) \longrightarrow \poe(U, p) \longrightarrow \poe(B, \pi(p))\longrightarrow 1.
\end{equation}
Let $q\circ v$ be a rank-$2$ Parshin chain that is not centered on $U$.  If $w \in \caV, T_{q\circ w}$ is a nontrivial subgroup of the first term of the short exact sequence; otherwise, $T_{q\circ w}$ projects to a nontrivial subgroup of $\poe(B)$.  In either case, $T_{q\circ w}$ is nontrivial, so $\caA(\caD(U)) \subseteq \iota(U)$.
\item The fibration short exact sequence and \autoref{thm:geometrictheoryofdecompositionandinertia} gives for any $v \in \caV$ that \begin{equation}D_{v} = \ker \pi_{*}: \poe(U)\longrightarrow \poe(B).\end{equation}
\item By \autoref{thm:geometrictheoryofdecompositionandinertia}, \begin{equation}\poe(|v|, p) = D_{v}\end{equation} in the short exact sequence \ref{ses:fibersequenceforvisibleaffine}, and the desired statement follows.
\item As $\pi|_{|h|}$ is nonconstant, we can complete and we get a diagram \begin{equation} \mbox{\xymatrix{H\ar[d]^{\pi} \ar[r] & \overline{H} \ar[d]^{\overline{\pi}}  \\ 
B \ar[r] & \overline{B}. }}\end{equation}
where $\overline{\pi}$ is surjective, and branch points are isolated.  If $p \in \partial B$ then $t_{p}$ has inverse image a disjoint union of loops in $|h|$ and a choice of one such loop for each $p \in \partial B$ provides the necessary meridians of rank-$2$ Parshin chains by \autoref{thm:geometrictheoryofdecompositionandinertia}.
\end{enumerate}
\end{proof}
We can similarly define $I_{B}$\index{notation}{IB@$I_{B}$} and $g(B)$\index{notation}{gB@$g(B)$} for any geometric set having a horizontal-vertical decomposition.
\begin{theorem}\label{VisibleAffine}
Let $\caS$ be a geometric set of $F$ with a disjoint union decomposition\index{definitions}{Horizontal-vertical Decomposition}\index{notation}{H@$\caH$}\index{notation}{V@$\caV$}
\[
\caS = \caH \cup \caV,
\]
where we call $\caH$ the \textbf{horizontal} and $\caV$ the \textbf{vertical fibers}.  Then $\caM(\caS)$ is a visible affine of $F$ with horizontal-vertical decomposition $\caH\cup \caV$ if and only if it satisfies the following properties:
\begin{enumerate}
\item \textbf{Fullness}. Let $v \in \Par_{1}(F)$ and $v \notin \caS$.  Then either $\partial_{\caS \cup \{v\}}(\caS) \subseteq a(\caS)$ or $\#\partial_{\caS \cup \{v\}}(\caS) = 1$.
\item \textbf{Homeomorphicity of Fibers}. For $v \in \caS$ let \begin{equation}
a(v) = \caP(v) \cap a(\caS)\text{ and }\caA(v) = \caP(v) \cap \caA(\caS).
\end{equation}
Then for any $v_{1}, v_{2} \in \caV$,
\begin{equation}
\#a(v_{1}) = \#a(v_{2})
\end{equation}
and
\begin{equation} g(v_{1}) = g(v_{2})\end{equation}
\item \textbf{Disjointness of Fibers}. $\caA(\caS) = \coprod_{v \in \caV} \caA(v),$ (a disjoint union).  Furthermore, there exists a geometric set $\caS' \supseteq \caS$ such that in $\caS',$ for any $v_{1}$ and $v_{2}$ distinct elements of $\caV,$ and any $p_{1} \in \Par_{2}(v_{1})$ and $p_{2} \in \Par_{2}(v_{2})$,
\begin{equation}
[p_{1}] \neq [p_{2}].
\end{equation}
Such an $\caS'$ is called \textbf{fiber-separating}.
\item \textbf{Hyperbolicity of Base}.  The base has at least three punctures; that is, the $\hat{\mathbb{Z}}$-rank of $I_{B}$ is $\geq 2$.
\item \textbf{Numerical Equivalence of Fibers}. For $h \in \caH$ and $v\in \caV$, let
\[
S_{h}(v) = \sum_{\substack{p\circ v \in \Par_{2}(v) \text{ s.t.} \\ [p\circ v] \in \caA(\caS)}}(p, h\cdot v; \caS).
\]
For every $h \in \caH$ there exists $n_{h} \in \mathbb{N}$ and a finite subset $\Sigma_{h}\subseteq \caV$ such that for all $v\in \caV$,
\[
S_{h}(v) \leq n_{h}
\]
with inequality strict only at $\Sigma_{h}$, and \[\bigcap_{h \in H}\Sigma_{h} = \emptyset\] for any cofinite subset $H\subseteq \caH$.
\item \textbf{Triviality of Monodromy}. Let $v \in \caV$.  Then $T_{v}^{a}$ is torsion-free in $\Pi_{\caS \setminus \{v\}}^{\ab}$, and the action of $T_{v}$ by conjugation on any $D_{v'}$ for $v' \in \caV$ is inner in $\Pi_{\caS \setminus \{v\}}$.
\item \textbf{Inheritance}.  Let $\caV'$ be any cofinite subset of $\caV$.  Then all the above properties hold for $\caV' \cup \caH$. 
\end{enumerate}
In this case, $\caS$ will be called a \textbf{visible affine geometric set}.
\end{theorem}
\begin{proof}
It is straightforward that every visible affine satisfies the hypotheses. Let
\[\overline{\caS} = \caD(\overline{\caM(\caS)})\]
where $\overline{\caM(\caS)}$ is a smooth compactification of $\caM(\caS)$ which is also fiber-separating.  Such an $\overline{\caS}$ exists by applying resolution of singularities to a compactification of $\caM(\caT)$ for any fiber-separating $\caT$; such a $\caT$ exists by Disjointness of Fibers.

We define two subsets of $\caP(\caS)$:
\begin{equation}
\beta_{i} = \partial_{\overline{\caS}}(\caS) \cap a(\caS)\text{ and }\beta_{e} = \partial_{\overline{\caS}}(\caS) \cap \caA(\caS). 
\end{equation}
We define the \textbf{vertical support} of $\partial_{\overline{\caS}}(\caS)$ to be 
\begin{equation}
\Delta(\beta_{e}) = \bigcup_{p \in \beta_{e}} \Delta(p) \cap \caV.
\end{equation}
By fullness, $\Delta(\beta_{e})$ is finite.

We define \begin{equation}\caV' = \caV \setminus \Delta(\beta_{e}).\end{equation} Then $\caS' = \caV'\cup \caH$ is a geometric set with horizontal-vertical decomposition, by Inheritance.

If $C, D$ and $E$ are divisors on $\overline{\caM(\caS)}$, we define
\begin{equation}
C\sim_{E} D \text{ if and only if } (C\cdot E) = (D \cdot E),
\end{equation} where we use the total intersection product as in \autoref{def:totalintersectionproduct}.
If $E_{1}, \dots, E_{n}$ generates $\NS(X)$, then we see that $C$ is numerically equivalent to $D$ if and only if $C\sim_{E_{i}} D$ for every $i = 1, \dots, n$.  If $C$ and $D$ are numerically equivalent, effective and pairwise disjoint, they are algebraically equivalent by \cite{IntersectionTheory}[19.3.1].  

Thus, to prove that $\caM(\caS)$ is a visible affine open, we need to prove that there exists a cofinite subset $\caV'' \subseteq \caV'$ and a finite set $\{h_{1}, \dots, h_{n}\} \subset \caH$ such that:
\begin{enumerate}
\item $|h_{i}|$ generate $\NS(\overline{\caM(\caS)})\otimes \mathbb{Q}$.
\item For all $v_{1}, v_{2} \in \caV''$, we have
\begin{equation}
|v_{1}| \sim_{|h_{i}|} |v_{2}|
\end{equation}
for each $1 \leq i \leq n$.
\end{enumerate}
Once we know this, we will know that the $\caM(\caV'' \cup \caH)$ is indeed an affine open, for the divisors each vary in an algebraic family.   Triviality of Monodromy allows us to use \cite[Theorem 0.8]{TamagawaGrothAffCurves} to ``plug the holes'' and deduce that $\caM(\caS)$ itself is indeed a visible affine open.  It is here that we use homeomorphicity of fibers to make sure we're ``plugging the holes'' with the right divisors.

$\NS(\overline{\caM(\caS)})\otimes \mathbb{Q}$ is spanned by very ample, prime divisors.  As the $|v|$ with $v \in \caV'$ are mutually disjoint, they cannot be very ample, as each very ample divisor intersects every other divisor.  Thus, all very ample divisors must be horizontal.  We thus can choose $h_{1}, \dots, h_{n} \in \caH$ so that $\{|h_{i}|\}$ generates $\NS(\overline{\caM(\caS)})\otimes \mathbb{Q}$.  For each $h_{i}$ we have
\begin{equation}
\left(|h_{i}|\cdot \bigcup_{b \in \overline{\caS} \setminus \caS} |b|\right) < \infty.
\end{equation}
Thus, by separation of points, there is at most a finite subset $\sigma'_{i}\subset \caV'$ for which if $s \in \sigma'_{i}$ there is an intersection between $|s|$ and $|h_{i}|$ at a point in $\beta_{i}$, and we may take
\[
\sigma_{i} = \sigma'_{i} \cup (\Sigma_{h_{i}} \cap \caV').
\]
But for each $v_{1}, v_{2} \in \caV' \setminus \sigma_{i}$, we have by \autoref{thm:localanabelianintersectiontheorem},
\[
|v_{1}| \sim_{h_{i}} |v_{2}|.
\]
Thus, we may take
\begin{equation}
\caV'' = \caV' \setminus \bigcup_{i = 1}^{n} \sigma_{i}.
\end{equation}
\end{proof}
\begin{proposition}\label{prop:affineopens}
Let $U$ be a visible affine and $X$ a maximal smooth model of $F$.  Then there is an open immersion 
\[
U\longrightarrow X
\]
under $\Spec F$ if and only if
\begin{equation}
\caD(U)\subseteq \caD(X)
\end{equation}
and
\begin{equation}
\caA(\caD(U)) = \caP_{\caD(X)}(\caD(U)).
\end{equation}
\end{proposition}
\begin{proof}
There is a birational map
\begin{equation}
\mbox{\xymatrix{
U \ar@{-->}[r] & X
}}
\end{equation}
defined outside a set of codimension $2$.  The minimal such exceptional set is, however, exactly $\partial_{\caD(X)}(\caD(U))$, so this arrow extends to a regular map if and only if $\partial_{\caD(X)}(\caD(U))$ is empty or, equivalently, $\caA(\caD(U)) = \caP_{\caD(X)}(\caD(U)).$
\end{proof}
As immediate corollaries we have:
\begin{corollary}\label{cor:detectpoints}
Let $\caS$ be a geometric set.  Then a point $p\in \caP(\caS)$ is in the image of $\iota$ if and only if there is a visible affine geometric set $\caS'$ such that $p \in \caP_{\caS}(\caS')$.  This is a group-theoretic criterion, and we will call these points, as group-theoretic objects, \textbf{geometric points} \index{definitions}{Geometric points}and denote the collection of all of them by $\caP^{\geom}(\caS)$\index{notation}{Pgeom@$\caP^{\geom}$}; $\caP^{\geom}(\caS)$ is identified by $\iota$ with $\caM(\caS)$.
\end{corollary}
\begin{corollary}
Given $\caS$, there is a group-theoretical recipe to determine whether $\caM(\caS)$ is proper.
\end{corollary}
\begin{definition}\label{def:propergeometricset}\index{definitions}{Geometric set!proper}
A geometric set $\caS$ such that $\caM(\caS)$ is proper will be called itself \textbf{proper}.
\end{definition}
\begin{definition}\label{def:ordering} We define a partial ordering $\preceq$ on $\caGeom(F)$ by saying that
\[
\caS\preceq \caS'
\]
if the following two conditions hold:
\begin{enumerate}
\item $\caS\subseteq \caS'$.
\item $\caP^{\geom}(\caS) \subseteq \caP_{\caS'}(\caS).$
\end{enumerate}
The category formed by this partial ordering (so a morphism $\varphi: \caS\longrightarrow \caS'$ is the relation $\caS\preceq \caS'$) is denoted by $\GBir_{\max}(F)$\index{notation}{GBirmax@$\GBir_{\max}(F)$}. The maximal smooth model $\caM$ thus extends uniquely to a functor
\[
\caM: \GBir_{\max}(F)\longrightarrow \Bir(F)
\]
and the set of prime divisors likewise extends
\[
\caD: \Bir(F)\longrightarrow \GBir_{\max}(F).
\]
\end{definition}
\begin{corollary}
$\caM$ is fully faithful.  The functors
\begin{center}
\mbox{\xymatrix{
\GBir_{\max}(F) \ar@<+3pt>[r]^-{\caM} & \Bir(F) \ar@<+3pt>[l]^-{\caD}
}}
\end{center}
form an adjoint pair, with $\caM$ right-adjoint to $\caD$. 
\end{corollary}
\section{Algebraic, Numerical, and Linear Equivalence of Divisors}
In this section, $\caS$ will denote a \textit{proper} geometric set.
The \textbf{divisor group} \index{definitions}{Divisor group} $\caDiv(\caS)$ \index{notation}{DivS@$\caDiv(\caS)$} is defined to be the free abelian group generated by $\caS$.
\begin{definition}
\begin{enumerate}
\item We call an element $\sum a_{i} v_{i} \in \caDiv(\caS)$  \textbf{effective}\index{definitions}{Divisor!effective} if and only if each $a_{i}\geq 0$, and we denote this by $D \geq 0$.  If $D\geq 0$ and $D \neq 0$ then we write $D > 0$\index{notation}{>}.  We also define a preorder on the divisors by:
\[
D \geq(\text{resp. }>)D' \Longleftrightarrow D - D' \geq(\text{resp. }>)\,0.
\]
\item The \textbf{support} \index{definitions}{Support} of a divisor $D$, denoted $\supp(D)$\index{notation}{suppD@$\supp(D)$}, is the collection of $v \in \caS$ such that the coefficient of $v$ in $D$ is nonzero.  
\item Given a divisor $D \in \caDiv(\caS)$ we may write $D$ uniquely as\index{notation}{Dmin@$D_{-}$}\index{notation}{Dplus@$D_{+}$}
\[
D = D_{+} - D_{-}
\]
where $D_{+}$ and $D_{-}$ are effective divisors, and $\supp(D_{+}) \cap \supp(D_{-}) = \emptyset$.
\end{enumerate}
\end{definition}
It is clear that
\begin{proposition}
The map\index{notation}{mu@$\mu$}
\[
\mu: \caDiv(\caS)\longrightarrow \caDiv(\caM(\caS))
\]
given by
\[
\mu\left(\sum_i a_i v_i\right) = \sum_i a_i |v_i|
\]
is an isomorphism.
\end{proposition}
Let $v_1$ and $v_2$ be two distinct prime divisors.  We define the \textbf{intersection pairing}\index{definitions}{Intersection pairing}\index{notation}{intv1v2@$(v_{1}\cdot v_{2})$} to be
\[
(v_1\cdot v_2) = \sum_{p \circ v_2 \in \Par_{2}(v_{2})} (p, v_1\cdot v_2; \caS),
\]
By Theorem \ref{thm:localanabelianintersectiontheorem},
\begin{proposition}
The intersection pairing $(v_1 \cdot v_2)$ coincides under pushforward with the intersection pairing on $\caM(\caS)$ when $v_{1} \neq v_{2}$ and otherwise extends by linearity to give self-intersection on $\caDiv(\caS)$.
\end{proposition}

Let $\caF\subset \caS$ be a visible affine with horizontal-vertical decomposition $\caF = \caH \cup \caV$, and 
\begin{enumerate}
\item Let $p$ be a puncture of the base, and let $T_{p}$ be its inertia group; this is the divisible hull of the image of a corresponding inertia group in $\caF$.  Then if $v \in \caS \setminus \caF$, we say its \textbf{multiplicity at $p$}\index{definitions}{Multiplicity} is the index\index{notation}{mp@$m_{p}(v)$}
\[
m_{p}(v) = \left\{\begin{array}{ll} [T^{a}_{p}: \pi(T^{a}_{v})]& \text{ if }T^{a}_{p} \cap T^{a}_v \neq \{0\} \\
0 & \text{otherwise} \end{array}\right.
\]
in $(\Pi_{\caF}/D_{v})^{\ab},$ with this equal to zero if the two groups are disjoint.

\item The \textbf{complete family} \index{definitions}{Complete family} of $\caF$ will be the subset\index{notation}{family@$\family(\caF)$}
\[
\family(\caF) \subset \caDiv(\caS)
\]
given by
\[
\caV \cup \{\sum_{v \in \caS} m_{p_{i}}(v) v \mid p_{i}\text{ a puncture of the base}\}.
\]
\item
We define \textbf{group-theoretical algebraic equivalence} \index{definitions}{Algebraic equivalence, group-theoretical}to be the equivalence relation on $\caDiv(\caS)$ generated by $\family(\caF)$ for all visible affines $\caF$ and denote this by $\sim_{\alg}$.  We define \textbf{group-theoretical linear equivalence}\index{definitions}{Linear equivalence, group-theoretical} to be the equivalence relation generated by $\family(\caF)$ for all visible affines with base having trivial unramified fundamental group (that is, for genus $0$ base) and denote this by $\sim_{\lin}$.  Two divisors $D_1$ and $D_2$ are said to be \textbf{group-theoretically numerically equivalent}\index{definitions}{Numerical equivalence, group-theoretical} if and only if for any divisor $E$ we have $(D_1 \cdot E) = (D_2 \cdot E)$.  This equivalence relation is denoted by $\sim_{\num}$.
\item Let $D \in \caDiv(\caS)$ and $D > 0$.  Then we define $|D|$ to be  the set of effective divisors linearly equivalent to $D$.
\item Let $D \in \caDiv(\caS)$ and let
\begin{equation}
D = E - E'
\end{equation}
be some expression of $D$ as a difference of two effective divisors.
\newline
Then we define the \textbf{group-theoretical complete linear system}\index{definitions}{Complete linear system, group-theoretical} to be 
\begin{equation}\index{notation}{Dee@$\lvert D\rvert$}
|D| = \{D' - E' \mid D' \in |E| \text{ and } D' - E' \geq 0\}.
\end{equation}
\end{enumerate}
We see immediately:
\begin{proposition}
Let $\caS$ be a proper geometric set of prime divisors.  Then the pushforward of group-theoretical linear (respectively, algebraic and numerical) equivalence on $\caDiv(\caS)$ by $\mu$ induces linear (respectively, algebraic and numerical) equivalence on $\caDiv(\caM(\caS))$.
\end{proposition}
\begin{corollary}\label{linearsystems}
The group-theoretical complete linear systems $|D|$ coincide with complete linear systems $|\mu(D)|$ on $\caM(\caS)$ and form finite-dimensional projective spaces over $\overline{\mathbb{Q}}$, and the lines in this projective space are given by linear families.
\end{corollary}
In particular, the Picard and N\'eron-Severi groups of $\caM(\caS)$ are group-theoretical invariants of $(G_{F}, \caS)$.
\section{Local Geometry: Tangent Spaces}\label{localgeometry}
Let $\caS$ be a geometric set on a two-dimensional function field $F$.  As we work locally, we do not need properness.

\begin{definition} Let $p \in \caP^{\geom}(\caS)$ be a point.  Then \index{notation}{Deltasmooth@$\Delta^{s}(p)$}\begin{equation}\Delta^{s}(p) =_{\operatorname{def}} \{v \in \Delta(p)\mid v\text{ is smooth at }p \}.\end{equation}
\end{definition}
\begin{definition} Let $w \in \caS$, smooth at a rank-$2$ Parshin chain $q \circ w$ such that $[q\circ w] \in \caP^{\geom}(\caS)$ and let $p \in \Delta^{s}([q\circ w])$.  Then we say that $v$ and $w$ are \textbf{tangent to order $n$}\index{definitions}{Tangent} at $p$ if and only if the local intersection number
\begin{equation}
(p, v\cdot w; \caS) \geq n+1.
\end{equation}
\end{definition}
Because $|v|$ and $|w|$ are actually tangent to order $n$ at $p$, tangency to order $n$ forms an equivalence relation, which we call $\sim_{n-\tan}$, and we thus recover the projectivized jet space\index{definitions}{Projectivized jet space} \[
\mathbb{P}\caJ^{n}_{p} = \Delta^{s}(p)/\sim_{n-\tan},
\]
at $p$.  In particular,
\[
\mathbb{P}\caJ_{1} =\mathbb{P}T_{p},
\]
the projectivized tangent space \index{definitions}{Projectivized tangent space} to $\caM(\caS)$ at $p$.  As $\caM(\caS)$ is a smooth surface, $\mathbb{P}T_{p}$ is a projective line.
\section{Projective Embeddings and Projective Coordinate Rings}
We start with some basic projective geometry, as in Artin \cite{GeomAlg}.
Let $(X, L)$ be an abstract projective space, given as its set of points $X$ and a set of lines $L$.
\begin{enumerate}
\item A subset $Y \subseteq X$ is \textbf{linearly closed} \index{definitions}{Linear closure, in projective geometry} if for any two points $P, P' \in Y$ the line $\overline{PP'} \subseteq Y$.  The linearly closed sets are closed under intersection.  The \textbf{linear closure} of $Y$ in $X$ is the intersection of all linearly closed spaces which contain $Y$.  As $X$ is linearly closed, the linear closure always exists.  We denote the linear closure of the union of a collection of subsets $V_1, \dots, V_n \subseteq X$ by $\overline{V_1\cdots V_n}$.
\item A point $P \in X$ is said to be \textbf{linearly independent}\index{definitions}{Linear independence} of a subset $Y$ if and only if $P \notin \overline{Y}$.  In particular, we call a set $P_1, \dots, P_n\in X$ \textbf{linearly independent} if for any subset $M \subseteq \{1, \dots, n\}$ and any $k \notin M$ we have $\overline{(P_m)_{m\in M}}\subsetneq \overline{P_k(P_m)_{m\in M}}.$
\item The \textbf{dimension} of $X$\index{definitions}{Dimension!in projective geometry} is the cardinality of a maximal set of linearly independent points minus $1$, and is denoted $\dim X$ \index{notation}{dim@$\dim$}(this is possibly infinite).
\end{enumerate}
Let $\caS$ be a proper geometric set on a function field $F$ of dimension $2$. 
\begin{definition} We say a point $p \in \caP(\caS)$ is \textbf{supported} on a divisor $D \in \caDiv(\caS)$ if and only if $\Delta(p) \cap \supp(D)$ is nonempty\index{definitions}{Support of a set of divisors}\index{notation}{suppD@$\supp(D)$}.  We say a set $X\subseteq \caDiv(\caS)$ \textbf{separate points}\index{definitions}{Separates points} if and only if for any two points $p_1, p_2 \in \caP(\caS)$ there are two divisors $D_1, D_2 \in X$ such that $p_1 \in \supp(D_1), p_1 \notin \supp(D_2)$ and $p_2 \notin \supp(D_1), p_2 \in \supp(D_2)$.  Given a set $S\subseteq \caDiv(\caS)$, we will define its \textbf{support}
\[
\supp(S) = \bigcup_{E \in S} \supp(E).
\]
\end{definition}
\begin{definition}
A point $p\in \caP(\caS)$ is said to be in the \textbf{base locus}\index{definitions}{Base locus} (and is called a \textbf{base point}\index{definitions}{Base point}) of $X\subseteq \caDiv(\caS)$ if $p$ is supported on every element of $X$; a set without base locus is called \textbf{basepoint free}\index{definitions}{Basepoint free}.  
\end{definition}
Separating points is strictly stronger than basepoint free.
\begin{definition}
We say the linear system $|D|$ \textbf{separates tangent lines}\index{definitions}{Separates tangent lines} at $p$ if and only if  for any $\ell \in \mathbb{P}T_p, |D| \cap \ell$  is nonempty.
\end{definition}
\begin{definition}
We call a divisor $D$ \textbf{very ample} if $|D|$ separates points and tangent lines.
\end{definition}

For any divisor $D$ and any effective divisor $E$ there is an injective map
\[
\alpha: |D| \longrightarrow |D+E|
\]
given by adding $E$ to each effective divisor in $|D|$.  This invokes a map in general
\[
\alpha: |D| \times |E|\longrightarrow |D+E|
\]
and in our particular case,
\[
\alpha^n: \Sym^n(|D|)\longrightarrow |nD|
\]
as addition is symmetric.  We will now call a divisor $D$ \textbf{$n$-adequate}\index{definitions}{$n$-Adequate divisor} if the linear closure of $\alpha^n(\Sym^n(|D|))$ in $|nD|$ is all of $|nD|$.  Any very ample divisor is adequate; indeed, a very ample divisor gives relations on the projective coordinate ring, which is generated in the first dimension.

\begin{definition} We define a \textbf{projectivizing datum} \index{definitions}{Projectivizing datum} to be a quadruple $(\caS, D, V, \rho)$ where:
\begin{enumerate}
\item $\caS$ is a proper geometric set.
\item $D$ is a very ample divisor.
\item $V$ is a $\overline{\mathbb{Q}}$-vector space.
\item $\rho: |D|\longrightarrow \mathbb{P}V$ is an isomorphism.
\end{enumerate}
\end{definition}
We know by the fundamental theorem of projective geometry that this $\rho$ is determined up to a semilinear automorphism of $V$, and fixing a $\rho$ rids us of the indeterminacy.  Let $\ProjData(\caS)$ be the collection of projectivizing data for $\caS$.  The outer automorphisms of $G_F$ act on the collection of all projectivizing data, by translation of the corresponding $\caS, D,$ and $|D|$.
\begin{proposition}
Given a projectivizing data $(\caS, D, V, \rho)$, the maps
\[
\Sym^n(|D|)\longrightarrow |nD|
\]
induce a canonical isomorphism
\[
\Sym^{n}(\rho): \mathbb{P}(\Sym^n(V)/I_n) \simeq |nD|
\]
compatible with all $\alpha^{i}$, where $I_n\subset \Sym^n(V)$ is a vector subspace.
\end{proposition}

Finally:
\begin{theorem}\label{thm:geometricreconstruction}\index{definitions}{Geometric reconstruction}
A projectivizing datum $(\caS, D, V, \rho)$ gives $\caM(\caS)$ uniquely the structure of a smooth, projective $\overline{\mathbb{Q}}$-variety, so that
\[
\caM(\caS)_{/\overline{\mathbb{Q}}} \simeq \Proj\left(\bigoplus_{n\geq 0} \Sym^n(V)/I_n\right),
\]
which induces an isomorphism
\[
\eta: \operatorname{Frac}(\overline{\mathbb{Q}}(\caM(\caS))) \longrightarrow F,
\]
and a corresponding isomorphism
\[
\eta: G_F\longrightarrow G_{\overline{\mathbb{Q}}(\caM(\caS))}.
\]
which respects inertia and decomposition groups of divisors.
\end{theorem}
\section{The Proof of the Birational Anabelian Theorem For Surfaces}
We can now apply the theory developed above to prove \autoref{BirationalAnabelianTheoremForSurfaces}.
\label{pf:BiratAnabThm}
\begin{proof}[Proof of \autoref{BirationalAnabelianTheoremForSurfaces}]
Let $F$ be a field, finitely generated over $\overline{\mathbb{Q}}$ and of transcendence degree $2$.  There is a canonical, injective map
\begin{equation}
\varphi: \Aut(F)\longrightarrow \Out_{\cont}(G_{F}),
\end{equation}
as any automorphism of $F$ which fixes each of its prime divisors must be the trivial automorphism. We construct the inverse
\begin{equation}
\psi: \Out_{\cont}(G_{F})\rightarrow \Aut(F).
\end{equation}

Choose a projectivizing datum $(\caS, D, V, \rho)$ fixed by no non-trivial automorphisms of $F$.  By \autoref{thm:geometricreconstruction}, we have
\begin{equation}
\eta: \overline{\mathbb{Q}}[\caM(\caS) \setminus D]\rightarrow F
\end{equation}
which gives an injection from a finitely generated ring to its field of fractions; the automorphisms of $F$ then act simply transitively on this set, as $(\caS, D, V, \rho)$ is fixed by no non-trivial automorphisms of $F$; however, as $(\caS, D, V, \rho)$ are determined by group theory, $\Out_{\cont}(G_{F})$ acts on this set and this gives our section $\psi$.

We must now prove that every continuous outer automorphism $\zeta \in \Out_{\cont}(G_{F})$ for which $\psi(\zeta) = e$ is an inner automorphism.  Choose $\zeta'$ to be a genuine continuous  automorphism in the class of $\zeta$.

Let $\{L_n\}$ be a sequence of finite, Galois extensions of $F$ and \[\Gamma_n = \Gal(\overline{F}|L_n)\] which satisfy the following properties:
\begin{enumerate}
\item $\zeta(\Gamma_n) = \Gamma_n$.
\item $\bigcap_n \Gamma_n = \{e\}.$
\end{enumerate}
That such a filtration exists comes from the fact that there are only finitely many translates of any finite-index, closed subgroup of $G_{F}$, which follows from the fact that there is a group-theoretic recipe to detect ramification information, and that there are only finitely many covers of a given degree with prescribed ramification (which in turn follows from the fact that geometric fundamental groups in question are finitely-presented).  But we can reconstruct $L_n$ from $\Gamma_n$, so any class of $\zeta$ then gives us an action of $\mu$ on $\bigcup_n L_n = \overline{F}$ trivial on $F$, which shows that its action on $\Gamma_{n}$ is induced by conjugation by an element of $G_{F}/\Gamma_{n}$ and as 
\begin{equation}
G_{F} = \varprojlim G_{F}/\Gamma_{n},
\end{equation}
 $\zeta$ is inner, and $\psi$ is an isomorphism.
\end{proof}
\section{Application: Grothendieck-Teichm\"uller Theory}
\label{sec:GT}
Let $\caM_{0,5}$ be the moduli space of genus $0$ curves with $5$ distinct, marked, ordered points over $\overline{\mathbb{Q}}$.  There is a natural isomorphism
\begin{equation}
i: \caM_{0,5}\rightarrow \mathbb{P}^{2} \setminus \caL
\end{equation}
where, if $\mathbb{P}^{2}$ is given projective coordinates $X, Y, Z$, $\caL$ is the union of the six lines given by homogeneous equations $L_{X}: X = 0, L_{Y}: Y = 0, L_{Z}: Z=0, L_{XY}: X = Y, L_{XZ}: X = Z, L_{YZ}: Y = Z$.  $S^{3}$ acts on $\caM_{0,5}$ by permuting these coordinates.
Harbater and Schneps defined a subgroup of $\Out(\poe(\caM_{0,5}))$
\begin{equation}
\Out^{\#}_{5} = \left\{\alpha \in \Out(\poe(\caM_{0,5}))\,\left| \begin{array}{l}\alpha\text{ commutes with the action of }S_{3} \\ \\ \alpha\text{ sends generators of inertia subgroups of each }L_{X}\text{, etc.,} \\ \text{to generators of possibly conjugate inertia subgroups} \end{array}\right. \right\} .
\end{equation}
A special case of the Main Theorem of \cite{HarbaterSchneps} is that there is an isomorphism
\begin{equation}
\eta': \widehat{GT} \rightarrow \Out^{\#}_{5}.
\end{equation}
The $\eta$ of Formula \ref{etadefinition} is the composition of $\eta'$ with the inclusion of $\Out^{\#}_{5}$ into $\Out(\poe(\caM_{0,5}))$.
To prove \autoref{LiftingCondition} we need first a lemma about visible affine opens.
\begin{lemma}\label{visaffopengeomsets} Let $\caS_{1}$ and $\caS_{2}$ be visible affine geometric sets with \begin{equation}\ker \rho_{\emptyset\caS_{1}} = \ker\rho_{\emptyset\caS_{2}}.\end{equation}
Then \begin{equation}\caS_{1} = \caS_{2}.\end{equation}
\end{lemma}
\begin{proof}
We note that a divisor $v$ is not centered on $\caS_{i}$ if and only if $\rho_{\emptyset\caS_{i}}(T_{v})$ is nontrivial, so under our assumptions $v$ is not centered on $\caS_{1}$ if and only if it is not centered on $\caS_{2}$. Let \begin{equation}\caS_{1} = \caS_{2} \cup \{v_{i}\}_{i \in V} \setminus \{w_{i}\}_{i \in W},\end{equation}
for some (possibly empty) finite, disjoint sets $V$ and $W$. If $W$ were nonempty, then each $w_{i}$ would be  uncentered on $\caS_{1}$ so $\rho_{\emptyset\caS_{1}}(T_{w_{2}})$ must be nontrivial, contradicting that $\ker \rho_{\emptyset\caS_{1}} = \ker\rho_{\emptyset\caS_{2}}$.  As any divisor complement of a visible affine adds generators to its fundamental group, $\caS_{1} = \caS_{2}$.
\end{proof}
\begin{proof}[Proof of \autoref{LiftingCondition}]
$\alpha$ as given in the hypotheses of the theorem satisfies the defining conditions of $\Out^{\#}_{5}$, by the Main Theorem of Harbater and Schneps.  Let $\alpha$ also satisfy the lifting condition of the theorem.  Then by the birational anabelian theorem for surfaces, $\tilde{\alpha} = \varphi(\beta)$ for some $\beta \in \Aut(\overline{\mathbb{Q}}(x,y))$.  If $\tilde{\alpha}$ preserves the kernel of $\gamma_{*}$, by \autoref{visaffopengeomsets}, $\tilde{\alpha}$ must preserve the unique geometric set $\caS$ which satisfies:
\begin{enumerate}
\item $\langle T_{v}\rangle_{v\in \caS} = \ker \gamma_{*}$.
\item $\caM(\caS)$ is isomorphic to $\caM_{0,5}$.
\end{enumerate}
This implies that $\tilde{\alpha}$ is such an automorphism of the affine $\caM_{0,5/\overline{\mathbb{Q}}}$, considered as a $\mathbb{Z}$-scheme (so that $\tilde{\alpha}$ could come from $G_{\mathbb{Q}}$ for example).  But 
\begin{equation}
\Aut(\caM_{0,5}) \simeq G_{\mathbb{Q}}\times S_{5},
\end{equation}
so the centralizer of $S_{5}$ is $G_{\mathbb{Q}}$, and $\alpha \in \Im \eta \circ \rho$.
\end{proof}

\section{Application: Absolute Galois Groups of Number Fields are Geometric Outer Automorphism Groups}\label{sec:AutomorphismGroups}
If a field $F$ satisfies the hypotheses of \autoref{outerauts}, then $\Aut(F) = G_{k}$, so \autoref{outerauts} follows immediately from \autoref{BirationalAnabelianTheoremForSurfaces}.

We can write down explicit examples of fields $F$ which satisfy the hypotheses of Theorem \ref{outerauts}.
We recall the following consequence of the main theorem from \cite{TrivAut}:
\begin{proposition}
Consider the affine curve
\[
C_\alpha: x^5 + y^{4} + \alpha xy + x = 0.
\]
Then for all but finitely many $\alpha \in \overline{\mathbb{Q}}$, this curve is nonsingular, hyperbolic, and has trivial automorphisms over its field of definition, which is $\mathbb{Q}(\alpha)$.
\end{proposition}
Let $C$ and $C'$ be two non-isomorphic, complete, hyperbolic curves over $\overline{\mathbb{Q}}$, with no $\overline{\mathbb{Q}}$-automorphisms, so that the compositum of their minimal fields of definition is $k$.  $C\times C'$ is the image of any variety birational to it under the canonical map (induced by the canonical bundle), so is a birational invariant.   As $C\times C'$ has no automorphisms over $\overline{\mathbb{Q}}$, there are no automorphisms of $\overline{\mathbb{Q}}(C\times C')$ over $\overline{\mathbb{Q}}$; otherwise they would have to act on the image of the canonical map.  Let $k = \mathbb{Q}(\alpha)$; for any number field $k$, there are infinitely many such $\alpha$, by Steinitz's theorem.  Then, the field $F = \overline{\mathbb{Q}}(C_1\times C_\alpha)$ satisfies the hypotheses of Theorem \ref{outerauts}, so \[\Out(G_F) \simeq G_k.\]

\section{Acknowledgements}
We would like to acknowledge Adam Topaz, Tam\'as Szamuely, Thomas Koberda, Glenn Stevens, Jay Pottharst, Dick Gross, Sinan \"Unver, Barry Mazur, Benson Farb, Anand Patel, David Massey, Mark Goresky, and K\"ur\c{s}at Aker for helpful conversations.  The influence of the ideas of Florian Pop on this paper --- as well as his personal encouragement of the author --- cannot be overstated.  He is also responsible for the terminology of ``birational'' versus ``geometric'' reconstruction.  Sacha Lochak provided invaluable comments on a draft of the paper, both typographical and substantial. We acknowledge also the hospitality of the Hausdorff Institute for Mathematics, the B\u{a}leanu family, and the Suboti\'{c} family while this paper was written.  The author was supported by a DoD NDSEG Fellowship and an NSF Graduate Fellowship. 
\bibliography{bibliography}
\end{document}

%% file: headerbeef.tex
\usepackage[utf8x]{inputenc}
\usepackage[T5]{fontenc}
\usepackage{graphicx}
\usepackage{amssymb}
\usepackage{amsmath}
\usepackage{amsfonts}
\usepackage{amsthm}
\usepackage{mathabx}
\usepackage{xypic}
\usepackage{ifthen}
\xyoption{all}
\usepackage{epstopdf}
\usepackage{times}
\usepackage{comment}
\usepackage{url}
\usepackage{float}
\usepackage{aliascnt}
\newcommand{\poe}{\pi_{1}^{\operatorname{\acute{e}t}}}

\newcommand{\caF}{\mathcal{F}}

\newcommand{\ab}{\operatorname{ab}}
\newcommand{\Isom}{\operatorname{Isom}}

\newcommand{\Out}{\operatorname{Out}}
\newcommand{\Aut}{\operatorname{Aut}}
\newcommand{\caO}{\mathcal{O}}
\newcommand{\caN}{\mathcal{N}}

\newcommand{\caD}{\mathcal{D}}
\newcommand{\caS}{\mathcal{S}}
\newcommand{\caM}{\mathcal{M}}
\newcommand{\caL}{\mathcal{L}}
\newcommand{\caT}{\mathcal{T}}

\newcommand{\caP}{\mathcal{P}}

\newcommand{\caA}{\mathcal{A}}
\newcommand{\caH}{\mathcal{H}}
\newcommand{\caV}{\mathcal{V}}

\newcommand{\caX}{\mathcal{X}}
\newcommand{\caE}{\mathcal{E}}
\newcommand{\caJ}{\mathcal{J}}

\newcommand{\caDiv}{\operatorname{Div}}

\newcommand{\Gal}{\operatorname{Gal}}
\newcommand{\Spec}{\operatorname{Spec}}

\newcommand{\ProjData}{\mathbf{ProjData}}

\newcommand{\Bir}{\mathfrak{Bir}}
\newcommand{\GBir}{\mathfrak{GBir}}
\newcommand{\trdeg}{\operatorname{tr.\,deg.}}
\newcommand{\NS}{\operatorname{NS}}

\newcommand{\Isomext}{\operatorname{Isom}^{{\Out}}}
\newcommand{\anal}{\operatorname{an}}
\renewcommand{\top}{\operatorname{top}}
\newcommand{\Bl}{\operatorname{Bl}}
\newcommand{\caGeom}{\mathfrak{Geom}}

\newcommand{\geom}{\operatorname{geom}}
\newcommand{\Sym}{\operatorname{Sym}}

\newcommand{\family}{\mathbf{Fam}}
\newcommand{\supp}{\mathbf{supp}}
\newcommand{\num}{\operatorname{num}}
\newcommand{\lin}{\operatorname{lin}}
\newcommand{\Proj}{\operatorname{Proj}}

\newcommand{\alg}{\operatorname{alg}}
\newcommand{\cont}{\operatorname{cont}}
\newcommand{\Par}{\operatorname{Par}}
\newcommand{\Lim}{\operatorname{Lim}}
\newcommand{\rk}{\operatorname{rk}}
\newcommand{\im}{\operatorname{im}}

\newcommand{\eqdef}{=_{\operatorname{def}}}

\newtheorem{theorem}{Theorem}
\newaliascnt{lemma}{theorem}  
\newtheorem{lemma}[lemma]{Lemma}  
\aliascntresetthe{lemma}  
 
 \newaliascnt{proposition}{theorem}  
\newtheorem{proposition}[proposition]{Proposition}  
\aliascntresetthe{proposition}  
  
 \newaliascnt{definition}{theorem}  
\newtheorem{definition}[definition]{Definition}  
\aliascntresetthe{definition}

 \newaliascnt{corollary}{theorem}  
\newtheorem{corollary}[corollary]{Corollary}  
\aliascntresetthe{corollary}  
  
 \newaliascnt{example}{theorem}  
\newtheorem{example}[example]{Example}  
\aliascntresetthe{example}  
  
 \newaliascnt{conjecture}{theorem}  
  
\aliascntresetthe{conjecture}  
  
 \newaliascnt{problem}{theorem}  
  
\aliascntresetthe{problem}  
  
\newaliascnt{question}{theorem}

\aliascntresetthe{question}

\newaliascnt{formula}{theorem}
\newenvironment{formula}{\begin{equation}}{\end{equation}}
\aliascntresetthe{formula}